\documentclass[12pt,twoside,letter]{amsart}
\usepackage{mathrsfs}

\usepackage{txfonts}
\usepackage{bbm}

\usepackage{amsfonts}
\usepackage{amsmath}
\usepackage{amssymb}
\usepackage{wasysym}
\usepackage{psfrag}
\usepackage{graphics,epsfig,amsmath}  
\usepackage{comment}
\usepackage{xcolor}
\usepackage{enumerate}

\newtheorem{theorem}{Theorem}[section]
\newtheorem{definition}[theorem]{Definition}
\newtheorem{remark}[theorem]{Remark}
\newtheorem{lemma}[theorem]{Lemma}
\newtheorem{prop}[theorem]{Proposition}
\newtheorem{corollary}[theorem]{Corollary}
\newtheorem{problem}[theorem]{Problem}
\newtheorem{example}[theorem]{Example}

\title[Assigned distance-like functions]{Set-assigned distance-like functions and the induced quotient metric space}

\author[X. Cui]{Xiaojun Cui}
\address{Department of Mathematics, Nanjing University, Nanjing 210093, China}
\email{xcui@nju.edu.cn}

\author[L. Jin]{Liang Jin}
\address{Department of Mathematics, Nanjing University of Science and Technology,
Nanjing 210094, China}
\email{jl@njust.edu.cn}

\author[X. Su] {Xifeng Su}
\address{School of Mathematical Sciences, Laboratory of Mathematics and Complex Systems (Ministry of Education)\\
Beijing Normal University\\
No. 19, XinJieKouWai St.,HaiDian District\\
 Beijing 100875, P. R. China}
\email{xfsu@bnu.edu.cn, billy3492@gmail.com}
\thanks{X. Cui is supported by the National Natural Science Foundation of China (Grant No. 12171234, 11631006, 11790272), the Project Funded by the Priority Academic Program Development of Jiangsu Higher Education Institutions (PAPD) and the Fundamental Research Funds for the Central Universities. L. Jin is supported by the National Natural Science Foundation of China (Grant No. 11901293) and Start-up Foundation of Nanjing University of Science and Technology (No. AE89991/114). X. Su is supported by the National Natural Science Foundation of China (Grant No. 11971060, 11871242).}

\begin{document}
\maketitle
\vspace{0.1in}


\begin{abstract}
 On a complete, connected, locally compact, non-compact geodesic space $(X,d)$, we assign each compact set a distance-like function. With the help of these functions, we obtain a pseudo-metric on the space of (non-empty) compact subsets of $X$ which is less than the Hausdorff distance. The quotient metric space is closely related to the large scale geometry of the ambient metric space. 
 
 In particular, we study both extreme cases --the pseudo-metric either vanishes or equals the original distance. The compactness of the level sets as well as stability under the Gromov-Hausdorff topology of such dl-functions are also investigated. As an application, we also give a representation formula of  any distance-like function in terms of the singleton-assigned distance-like functions defined here.
\end{abstract}

\section{Introduction}\label{sec:introduction}
Let $(X,d)$ be a complete connected, locally compact, non-compact geodesic space, i.e., a length space on which the distance between pair of points is attained by the length of some Lipschitz curve connecting them. Such a curve fits into the name \textbf{geodesic}, although the latter term admits a more wide meaning:
\begin{definition}[\cite{BBI}, Definition 2.5.27]\label{geo}
A continuous map $\gamma:I\rightarrow X$ is called a geodesic if for every $t\in I$ there exists a closed interval $J$ containing a neighborhood of $t$ in $I$ such that the length of $\gamma|_{J}$ attains the distance between its endpoints. In other words, a geodesic is a curve which is \textbf{locally} a distance minimizer.
\end{definition}
It is well known that the metric $d(\cdot,\cdot):X\times X\rightarrow\mathbb{R}, x\in X$ and the naturally defined geodesics are a source of many investigations of rich geometrical and topological information of $X$ \cite{G2}. In this paper, we focus on the global analogues of distance called distance-like functions (dl-functions for short), which also play an indispensible role in exploring the global picture of the metric space $(X,d)$. One could formulated the definition of dl-functions as
\begin{definition}\label{d-l}
$u:X\rightarrow\mathbb{R}$ is called a dl-function if there is a sequence of closed subsets $\{H_{n}\}_{n\in\mathbb{N}}$ of $X$ diverging to infinity (i.e., for some $x_{0}\in X,\,\,d(x_{0},H_{n})\rightarrow\infty$) and a sequence $\{c_{n}\}_{n\in\mathbb{N}}\subseteq\mathbb{R}$ such that
\begin{equation}\label{d-l def}
u(x):=\lim_{n\rightarrow\infty}[d(x,H_{n})-c_{n}],\quad\quad\text{for any}\,\,x\in X
\end{equation}
exists in the compact-open topology.
\end{definition}

\begin{remark}\label{1-Lip}
As a direct consequence of Definition~\ref{d-l}, any dl-function is a limit of 1-Lipschitz functions, thus is also 1-Lipschitz.
\end{remark}

\begin{remark}\label{eikonal}
If the geodesic space $(X,d)$ is a Riemannian manifold $(M,g)$, then it is proved in \cite{C} that the notion of dl-functions coincides with that of viscosity solutions of the eikonal equation
\begin{equation} \label{eik}
|\nabla u(x)|_{g}=1,\hspace{0.5cm}x\in M.
\end{equation}
More generally, in the context of the Hamilton-Jacobi equations, the theory of viscosity solutions is an intersection point of many interesting branches of applied mathematics. For instance, it has been established from the viewpoint of optimal control \cite{C-S}, Hamiltonian dynamical systems \cite{F} and PDEs \cite{L}.
\end{remark}

The structure of the space of dl-functions on $X$ gives us a global picture of $X$, especially some insights on the geometry near infinity of $X$. For instance, a subclass of such functions --the Busemann functions associated to a given geodesic ray, describes the limit behavior of the distance function (up to constants) to some point moving to infinity along a given direction and their level sets, usually called horospheres, could be regarded as the geodesic spheres prescribed by that geodesic ray and centered at some point belonging to infinity.

\medskip
In this paper, we will define another subclass of dl-functions, referred as the \textbf{set-assigned dl-functions}. It exhibits the similar large scale distance information as the Busemann function does, but in another aspect. Precisely, it describes the limit behavior of distance to the ``infinite sphere''. We believe that this class of dl-functions would provide us an opportunity to look at the infinity of the manifold from a quite different viewpoint. Also, it is hopeful that the constructions of set-assigned dl-functions will play roles in the study of various problems arising in metric structures, especially for certain special length spaces. We will come back to this issue in the future work such as applications in the transportation theory.

\medskip
Now let us go into the detail of the formal definition of set-assigned dl-functions. Recall that the distance function from a non-empty, closed set $H\subseteq X$ is defined as
\[
d(x,H)=\inf_{h\in H}d(x,h), \qquad \forall~x \in X.
\]
For every $x\in X, r>0$, we denote by $B_{r}(x)=\{x^{\prime}\in X\,|\,d(x,x^{\prime})<r\}$ the open geodesic ball centered at $x$ and $S\!_r(x)$ its boundary sphere.

\vspace{5pt}
For each compact subset $K_{0}\subset X$ and $r>0$, $B_{r}(K_{0})=\cup_{x\in K_0}B_{r}(x)$ is the $r$-ball around $K_0$ and $S\!_r(K_{0})=\partial B_{r}(K_{0})$, its topological boundary. From the definition, it is easily seen that
\begin{equation}\label{r-sphere}
S\!_r(K_0)=\{x\in X\,\,|\,\,d(x,K_0)=r\}.
\end{equation}
We use the notation
\begin{equation}\label{u-r}
u^{r}_{K_{0}}(x):=d(x, S\!_r(K_{0}))-r.
\end{equation}
By the assumptions on the space $(X,d)$ and some other elementary facts for distance functions, we could show that
\begin{theorem}\label{sa sol}
The limit
\begin{equation}\label{set-assigned def}
u_{K_0}(x):=\lim\limits_{r\rightarrow +\infty} u^{r}_{K_{0}}(x)
\end{equation}
exists for every $x\in X$, and thus defines a function $u_{K_0}:X\rightarrow\mathbb{R}$.
\end{theorem}

\begin{remark}\label{vis}
Let \,$X=(M,g)$ be a connected, complete, non-compact smooth Riemannian manifold. Then $u_{K_{0}}:M\rightarrow\mathbb{R}$ satisfies the eikonal equation \eqref{eik} in the viscosity sense, and is locally semi-concave (with linear modulus) on $M$.

In fact, $u_{K_0}$ is a dl-function on $M$. By Remark \ref{eikonal} or \cite[Theorem 1]{C}, $u_{K_{0}}:M\rightarrow\mathbb{R}$ is a viscosity solution of \eqref{eik} and its local semi-concavity (with linear modulus) is a direct consequence of the strict convexity of $g$ on each tangent space.
\end{remark}

Let $\mathcal{K}(X)$ denote the space of all non-empty compact subsets of $(X,d)$. Theorem ~\ref{sa sol} leads us to
\begin{definition}\label{def sa}
For any $K_{0}\in\mathcal{K}(X)$, the function $u_{K_{0}}$ defined above is called $K_{0}$-assigned dl-functions. $K_{0}$ is called the base-set of $u_{K_{0}}$ and the family of functions $\{u_{K_{0}}|K_{0}\in\mathcal{K}(X)\}$ are called set-assigned dl-functions.
\end{definition}

Set-assigned dl-functions are characterized by a fundamental property: being minimal within the class of all dl-functions nonnegative on the referred set.
\begin{theorem}\label{sa-Min}
For any $K_{0}\in\mathcal{K}(X)$ and any dl-function $u:X\rightarrow\mathbb{R}$ such that $u|_{K_0}\geq0$, we have
\[
u_{K_{0}}\leq u\quad\text{on}\quad X.
\]
\end{theorem}

For $K_0, K_1\in\mathcal{K}(X)$, if we set
\[
u_{K_0}(K_1)=\min_{x\in K_1}u_{K_0}(x).
\]
An important application of Theorem \ref{sa-Min} is the following construction of a pseudo metric on $\mathcal{K}(X)$.
\begin{theorem}\label{pm}
The function $\rho:\mathcal{K}(X)\times\mathcal{K}(X)\rightarrow[0,+\infty)$ defined as
\begin{equation}\label{pseudo metric}
\rho(K_0,K_1):=-\frac{1}{2}[u_{K_0}(K_1)+u_{K_1}(K_0)]
\end{equation}
is a pseudo metric on $\mathcal{K}(X)$. Denote by $\sim$ the equivalence relation on $\mathcal{K}(X)$ induced by $\rho$, i.e., $K_0\sim K_1$ if $\rho(K_0,K_1)=0$.

\vspace{1em}
Moreover, $K_0\sim K_1$ if and only if there exists $c\in\mathbb{R}$ such that
\[
u_{K_0}=u_{K_1}+c\quad\text{ on } X.
\]
\end{theorem}

\begin{definition}\label{def sa}
$x_{0}\in K_{0}$ is called an exposed point of $K_0$ if $u_{K_{0}}(x_{0})=0$, and we denote by ep$(K_0)$ the set of all exposed points of $K_0$. 
\end{definition}

If we interpret ep as the projection from infinite sphere to the base set, the following theorem seems reasonable. 
\begin{theorem}\label{ep}
For any $K_{0}\in\mathcal{K}(X)$, we have
\begin{enumerate}[$(i)$]
  \item  ep$(K_0)$ is a nonempty subset of $\partial K_0$, where $\partial K_0$ denotes the topological boundary of $K_0$, and ep\,$\circ$\,ep$(K_0)\,=\,$ep$(K_0)$.
  \item  $u_{K_0}=u_{\text{ep}(K_0)}$. In other words, $K_0\sim$ ep$(K_0)$.
\end{enumerate}
\end{theorem}
Thus, we denote by $ep:\mathcal{K}(X)\rightarrow\mathcal{K}(X)$ the operator taking exposed points of non-empty compact subsets of $X$.

We shall use the notation $d_{sa}$ to denote the quotient metric on $\mathcal{K}(X)/\!\!\sim$ induced  by $\rho$, and call it the \textbf{set-assigned metric}. It will be shown later that the metric space $(\mathcal{K}(X)/\!\!\sim,d_{sa})$  is closely related to the large scale geometry of $(X,d)$.

\medskip
We call a  complete connected, locally compact, non-compact geodesic space associated with a distinguished compact subset  the \textbf{set-assigned geodesic space}.
\begin{definition}\label{GH metric}
The Gromov-Hausdorff metric between two set-assigned geodesic spaces $(X_0, d_0, K_0)$ and  $(X_1, d_1, K_1)$ is defined as
$$d_{GH}((X_0, d_0, K_0),(X_1, d_1, K_1))=\inf_{d}\,\,\{d_H(X_0,X_1)+d_H(K_0,K_1)\},$$
where $\inf$ is taken over all admissible metric $d$ on the disjoint union $X_0\sqcup X_1$, i.e., the metric $d: (X_0\sqcup X_1)\times(X_0\sqcup X_1)\rightarrow\mathbb{R}_{+}$ satisfies $d|_{X_0\times X_0}=d_0,d|_{X_1\times X_1}=d_1$, and $d_{H}$ always denotes the Hausdorff metric associated to the admissible metric $d$.
\end{definition}

\begin{remark}
Since the spaces $X_i, i=0,1$ are non-compact, the Gromov-Hausdorff metric $d_{GH}$ defined above may be infinite for many cases.
\end{remark}


The Gromov-Hausdorff metric describes the similarity of metric structure between two set-assigned geodesic spaces and it is closely related to the existence of the following maps.

\begin{definition}\label{e-iso}
Let $(X_0, d_0, K_0),(X_1, d_1, K_1)$ be two set-assigned geodesic spaces. The map $f:(X_0,K_0)\rightarrow(X_1,K_1)$ is called a set-assigned $\epsilon$-isometry if
\begin{enumerate}
  \item $\text{ dis }f:=\sup_{(x,x')\in X_0\times X_0}\{\,|d_0(x,x')-d_1(f(x),f(x'))|\,\}<\epsilon$,
  \item $f(X_0)$ is an $\epsilon$-net in $X_1$ and $f(K_0)$ is an $\epsilon$-net in $K_1$.
\end{enumerate}
\end{definition}

\begin{remark}
We do not impose any continuity assumption on $f$.
\end{remark}

We shall show that the set-assigned dl-functions are continuous objects with respect to topology of the set-assigned geodesic spaces induced by the Gromov-Hausdorff metric defined in Definition~\ref{GH metric}.
\begin{theorem}\label{sa-GH}
Let $(X_0, d_0, K_0)$ and $(X_1, d_1, K_1)$ be two set-assigned geodesic spaces with
$$
d_{GH}((X_0, d_0, K_0),(X_1, d_1, K_1))<\epsilon,
$$
then there exists a set-assigned $2\epsilon$-isometry $f:(X_0,K_0)\rightarrow(X_1,K_1)$ with
\begin{equation}
|\,u_{K_0}-u_{K_1}\circ f\,|\,\,(x_0)<10\epsilon\quad\text{for any}\quad x_0\in X_0.
\end{equation}
\end{theorem}



The remaining of this paper is organized as follows. In Section~\ref{sec:Gradient line}, we recall some useful notions and facts related to dl-functions as preliminaries. Section 3 consists of the main part of this paper, where we prove Theorems~\ref{sa sol},\ref{sa-Min},\ref{pm},\ref{ep} listed above. As a consequence, the pseudo metric space $(\mathcal{K}(X),\rho)$ and its quotient metric space are defined. Section 4 focuses on the proof of Theorem \ref{sa-GH}. As the final part, Section 5 offers results and examples under the additional assumption that the base-set is a singleton.

\section{Dl-functions and their gradient lines}\label{sec:Gradient line}
In this section, we recall elementary facts about dl-functions on which later discussions are based. Particular attentions are paid to a description of gradient lines of dl-functions. We also give a representation formula for any dl-functions in terms of its subclass--Busemann functions.

\subsection{Negative gradient rays of dl-functions}
It is well known that on a complete, non-compact Riemannian manifold, co-rays of a given ray $\gamma$ are \textbf{negative gradient rays} of the associated Busemann function $b_{\gamma}$ and are constructed by standard methods. As an application of these ideas, we define the notion of negative gradient rays of any dl-functions on the more general metric spaces.

In this section, we consider that $X$ is a complete, locally compact, non-compact geodesic spaces. For $x\in X$ and a closed set $H\subseteq X$, thanks to the locally compactness of $X$, there is $h\in H$ such that $d(x,h)=d(x,H)$. Such an $h$ is called a foot of $x$ on $H$. Since $X$ is a geodesic space, there is a unit speed minimal geodesic segment $\gamma:[0,d(x,h)]\rightarrow X$ connecting $x$ and $h$, which means that for any $[t_{1},t_{2}]\subseteq[0,d(x,h)]$, the length
\begin{equation}\label{min}
L(\gamma|_{[t_{1},t_{2}]})=d(\gamma(t_{1}),\gamma(t_{2}))=t_{2}-t_1.
\end{equation}
For the notion of \emph{speed of a curve} on a metric space, see \cite[Page 55, Definition 2.7.1]{BBI}.

\medskip
Now we assume that $u:X\rightarrow\mathbb{R}$ is a dl-function and $\{H_{n}\}_{n\in\mathbb{N}}$ is the sequence of closed subsets of $X$ arising in the Definition \ref{d-l} of $u$, then for any $x\in X$, choose $x_{n}\rightarrow x$ on $X$ as well as unit speed minimal geodesic segments $\gamma_{n}:[0,r_{n}]\rightarrow X$ with $r_{n}=d(x_{n},H_{n}),\,\,\gamma_{n}(0)=x_{n}\,$ and $\,\gamma_{n}(r_{n})$ a foot of $x_{n}$ on $H_{n}$. By Ascoli-Arzel\`a Theorem, up to a subsequence, $\gamma_{n}$ converges uniformly on compact intervals to some $\gamma:[0,+\infty)\rightarrow X$, which we call a negative gradient ray of $u$, or simply a \textbf{$u$-ray}. It is clear that $\gamma$ is a distance-minimizer between any pair of its points, thus is a unit speed geodesic ray. We first verify that the name ``negative gradient ray of $u$" constructed in the above manner matches the reality.

\begin{prop}\label{gradient}
Let $u:X\rightarrow\mathbb{R}$ be a dl-function and $\gamma:[0,+\infty)\rightarrow X$  be one of its negative gradient ray constructed above, then for $\,0\leq t_{1}\leq t_{2},$
\begin{equation}\label{grad}
u(\gamma(t_{2}))-u(\gamma(t_{1}))=t_{1}-t_{2}.
\end{equation}
\end{prop}

\begin{proof}
Let $\gamma_{n}:[0,r_{n}]\rightarrow X$ be the sequence of minimal geodesic segments used in the construction of $\gamma$. Without loss of generality, we assume that for any $n\in\mathbb{N}, t_{2}\leq r_{n}$. By Definition \ref{d-l}, we have
\begin{align*}
&u(\gamma(t_{2}))-u(\gamma(t_{1}))\\
=\,&\lim_{n\rightarrow\infty}[d(\gamma_{n}(t_{2}),H_{n})-d(\gamma_{n}(t_{1}),H_{n})]\\
=\,&\lim_{n\rightarrow\infty}[d(\gamma_{n}(t_{2}),\gamma_{n}(r_{n}))-d(\gamma_{n}(t_{1}),\gamma_{n}(r_{n}))]\\
=\,&\lim_{n\rightarrow\infty}[(r_{n}-t_{2})-(r_{n}-t_{1})]=t_{1}-t_{2},
\end{align*}
where the second equality uses the fact that $\gamma_{n}(r_{n})$ is a foot of $\gamma_{n}(0)$ on $H_{n}$ and the third one uses \eqref{min}. This completes the proof.
\end{proof}

\begin{remark}\label{ray}
We observe that any unit speed curve $\gamma:[0,+\infty)\rightarrow X$ satisfying \eqref{grad} must be a geodesic ray: for any closed interval $[t_{1},t_{2}]\subseteq[0,+\infty)$,
$$
d(\gamma(t_{1}),\gamma(t_{2}))\geq u(\gamma(t_{1}))-u(\gamma(t_{2}))=t_{2}-t_{1}=L(\gamma|_{[t_{1},t_{2}]}),
$$
where the first inequality holds since $u$ is 1-Lipschitz (due to Remark~\ref{1-Lip}). Thus the above inequalities are equalities.
\end{remark}

In addition, we also note that
\begin{prop}\label{pro co-ray}
For a dl-function $u:X\rightarrow\mathbb{R}$,
\begin{enumerate}
  \item there is a $u$-ray starting from any $x\in X$,
  \item the limit of a sequence of $u$-rays is also a $u$-ray.
\end{enumerate}
\end{prop}

\begin{proof}
(1) can be easily seen from the construction of $u$-ray above, since one can choose $x_{n}=x$ for $n\geq1$.

To show (2), let $\gamma_{n}:[0,+\infty)\rightarrow X$ be a sequence of $u$-rays that converges to a ray $\gamma:[0,+\infty)\rightarrow X$. From the construction of $u$-rays, there is a sequence of minimal geodesic segments $\gamma_{n,m}:[0,r_{n,m}]\rightarrow X$ converging to $\gamma_{n}$ as $m\rightarrow\infty$ with $r_{n,m}=d(\gamma_{n,m}(0),H_{m})\rightarrow+\infty$ and $\gamma_{n,m}(r_{n,m})$ is a foot of $\gamma_{n,m}(0)$ on $H_{m}$. Denote by $m(n)\geq1$ the least integer such that $r_{n,m(n)}\geq n$ and
$$
d(\gamma_{n,m(n)}(t),\gamma_{n}(t))\leq\frac{1}{n},\quad\text{ for }t\in[0,n].
$$
So we have
\begin{equation*}
d(\gamma_{n,m(n)}(t),\gamma(t))\leq\frac{1}{n}+d(\gamma_{n}(t),\gamma(t)),\quad\text{ for }\,\,t\in[0,n].
\end{equation*}
This implies that the minimal geodesic segments $\gamma_{n,m(n)}$ converge uniformly to $\gamma$ on compact intervals.
\end{proof}

\medskip
Furthermore, if we assume the following non-branching property:
\begin{enumerate}[(NB)]
  \item for $a<t<b$, any two unit speed minimal geodesic segments $\xi,\eta:[a,b]\rightarrow X$ with $\xi(a)=\eta(a)$ and $\xi(t)=\eta(t)$ must coincide,
\end{enumerate}
then the following converse version of Proposition \ref{gradient} holds:
\begin{prop}\label{co-ray}
If $\gamma:[0,+\infty)\rightarrow X$ is a unit speed curve satisfying \eqref{grad}, then it is a $u$-ray.
\end{prop}

\begin{remark}
Comparing \cite[Page 36, 8.1]{B}, if a complete, locally compact, non-compact geodesic space $X$ satisfies (NB), then it is a G-space defined by H.Busemann, see \cite[Page 37]{B}. The assumption (NB) holds for any connected, complete, non-compact smooth Riemannian manifold.
\end{remark}

The main ingredient of Proposition~\ref{co-ray} is contained in the following
\begin{lemma}\label{uni co-ray}
Let $\gamma:[0,+\infty)\rightarrow X$ be a unit speed geodesic ray satisfying \eqref{grad}, then for any $t_{0}>0$, the $u$-ray starting from $\gamma(t_{0})$ is unique and thus has to be the sub-ray $\gamma|_{[t_{0},+\infty)}$.
\end{lemma}

\begin{proof}
Let $\xi:[0,+\infty)\rightarrow X$ be a unit speed geodesic ray satisfying \eqref{grad} with $\xi(0)=\gamma(t_{0})$. Consider the concatenate curve $\tilde{\gamma}:[0,+\infty)\rightarrow X$ defined as for $t\in[0,t_{0}],\tilde{\gamma}(t)=\gamma(t)$ and for $t\in[t_{0},+\infty),\tilde{\gamma}(t)=\xi(t-t_{0})$.

Now we show that $\tilde{\gamma}$ satisfying \eqref{grad}: without loss of generality,
we assume $0\leq t_{1}\leq t_{0}\leq t_{2}$, and we have
\begin{align*}
&u(\tilde{\gamma}(t_{2}))-u(\tilde{\gamma}(t_{1}))=[u(\tilde{\gamma}(t_{2}))-u(\tilde{\gamma}(t_{0}))]+[u(\tilde{\gamma}(t_{0}))-u(\tilde{\gamma}(t_{1}))]\\
=\,&[u(\xi(t_{2}))-u(\xi(t_{0}))]+[u(\gamma(t_{0}))-u(\gamma(t_{1}))]\\
=\,&(t_{0}-t_{2})+(t_{1}-t_{0})=t_{1}-t_{2}.
\end{align*}
By Remark \ref{ray}, $\tilde{\gamma}$ is a geodesic ray. Then assumption (NB) implies that $\tilde{\gamma}|_{[t_{0},+\infty)}\equiv\gamma|_{[t_{0},+\infty)}$. Thus $\xi(\cdot) =\gamma(t_{0}+\cdot)$ and it is the only unit speed geodesic ray defined on $[0,+\infty)$ satisfying \eqref{grad} and $\xi(0)=\gamma(t_{0})$. By the first item of Proposition~\ref{pro co-ray}, $\xi$ is also the unique $u$-ray starting from $\gamma(t_{0})$.
\end{proof}

\begin{proof}[Proof of Proposition \ref{co-ray}]
For any geodesic ray $\gamma:[0,+\infty)\rightarrow X$ starting from $x$ and satisfying \eqref{grad}, by Lemma \ref{uni co-ray} above, $\gamma|_{[\frac{1}{n},+\infty)}$ is a $u$-ray starting from $\gamma(\frac{1}{n})$. To complete the proof, we use the second item of Proposition~\ref{pro co-ray} to conclude that $\gamma$ must be a $u$-ray starting from $x_0$.
\end{proof}

\subsection{An equivalent definition of dl-functions}
We begin with another definition of dl-functions given by Gromov in \cite{G1}.
\begin{definition}\label{origin-def}
If a continuous function $u:X\rightarrow\mathbb{R}$ satisfies that for each $t\in\mathbb{R},\,x\in X$ with $u(x)\geq t$, we have
\begin{equation}\label{origin}
u(x)=t+d(x,u^{-1}(-\infty,t)),
\end{equation}
then $u$ is called a dl-function.
\end{definition}

\medskip
It is not difficult to show that this definition is equivalent to the one in Definition~\ref{d-l} at the beginning of the introduction.
\begin{prop}\label{equi}
$u:X\rightarrow\mathbb{R}$ is a dl-function in the sense of Definition~\ref{d-l} if and only if it satisfies Definition~\ref{origin-def}.
\end{prop}

\begin{proof}
If we take $H_{n}=u^{-1}(-\infty,-n]$ and $c_{n}=n$, it is clear that $u$ satisfying \eqref{origin} is a dl-function in the sense of Definition \ref{d-l}. Now let $u$ be a dl-function in the sense of Definition \ref{d-l}. By (i) of Proposition \ref{pro co-ray}, if $u(x)>t$, then there is a co-ray $\gamma:[0,+\infty)\rightarrow X$ starting from $x$, then by Proposition \ref{gradient}, we have
\begin{equation*}
u(\gamma(u(x)-t))=t,\quad d(x,\gamma(u(x)-t))=d(x,u^{-1}(-\infty,t)),
\end{equation*}
and $u(x)=u(\gamma(u(x)-t))+d(x,\gamma(u(x)-t))=t+d(x,u^{-1}(-\infty,t))$, where the second equality holds since $u$ is 1-Lipschitz.
\end{proof}

\medskip
Gromov's definition helps us to obtain the stability of dl-functions easily, as is shown by the following
\begin{prop}
Let $u_{n}:X\rightarrow\mathbb{R}$ be a sequence of dl-functions converging on compact subsets to a continuous function $u:X\rightarrow\mathbb{R}$, then $u$ is a dl-function.
\end{prop}

\begin{proof}
We use Gromov's definition of dl-functions. If $u(x)>t$, then for $n$ sufficiently large, $u_{n}(x)>t$ and
\begin{align*}
u_{n}(x)=t+d(x,u_{n}^{-1}(-\infty,t)).
\end{align*}
Since $u_{n}(x)$ converges, there are $h_{n}\in u_{n}^{-1}(t)$ such that $d(x,h_{n})=d(x,u_{n}^{-1}(-\infty,t))$ are uniformly bounded. Since $X$ is locally compact, there exists $h^{\ast}\in u^{-1}(t)$ such that $h_{n}\rightarrow h^{\ast}$, thus
$$
u(x)=\lim_{n\rightarrow\infty}u_{n}(x)=\lim_{n\rightarrow\infty}[t+d(x,h_{n})]=t+d(x,h^{\ast})\geq t+d(x,u^{-1}(-\infty,t)).
$$
On the other hand, since $u$ is 1-Lipschitz, for any $y\in u^{-1}(t), u(x)-t=u(x)-u(y)\leq d(x,y)$, then
\begin{align*}
u(x)\leq\inf_{y\in u^{-1}(t)}[t+d(x,y)]=t+d(x,u^{-1}(t))=t+d(x,u^{-1}(-\infty,t)),
\end{align*}
which completes the proof.
\end{proof}

\medskip

\subsection{A representation formula of dl-functions}\label{sec:representation formula}
Before ending this section, we introduce an independent result --a representation formula for any dl-functions by Busemann functions. For convenience, we recall the formal definition of Busemann functions here.
\begin{definition}\label{b}
Let $\gamma:[0,+\infty)\rightarrow X$ be any unit speed geodesic ray starting from $x_{0}$ ($\gamma$ always exists by the assumptions imposed on $X$), and the function $b_{\gamma}:X\rightarrow\mathbb{R}$ \,defined as
\begin{equation*}
b_{\gamma}(x):=\lim_{t\rightarrow+\infty}[d(x,\gamma(t))-t]
\end{equation*}
is called the Busemann function associated to $\gamma$.
\end{definition}
Note that the Busemann functions are the prototype of all dl-functions. Indeed, for a unit speed geodesic ray $\gamma:[0,\infty)\rightarrow X$, by taking $H_{n}=\{\gamma(n)\},c_{n}=n$ in Definition~\ref{d-l}, one obtains the Busemann functions associated to $\gamma$. Let $\mathscr{R}(u)$ denote the set of all $u$-rays, then
\begin{prop}\label{representation formula}
For any dl-function $u:X\rightarrow\mathbb{R}$ and $x\in X$,
\begin{equation*}
u(x)=\inf_{\gamma\in\mathscr{R}(u)}[u(\gamma(0))+b_{\gamma}(x)].
\end{equation*}
\end{prop}

\begin{proof}
By the definition of $u$, we have for $t\geq0$,
\begin{align*}
u(x)=& \lim_{n\rightarrow\infty}[d(x,H_{n})-c_{n}]\\
\leq\,&\lim_{n\rightarrow\infty}[d(x,\gamma(t))+d(\gamma(t),H_{n})-c_{n}]\\
=\,&[d(x,\gamma(t))-t]+ \lim_{n\rightarrow\infty} [t+d(\gamma(t),H_{n})-c_{n}]\\
=\,&[d(x,\gamma(t))-t]+[t+u(\gamma(t))].  \\
\end{align*}
Since $\gamma$ is a $u$-ray, it follows that $u(\gamma(0))-u(\gamma(t))=t$. Thus the above inequality reads as $u(x)\leq[d(x,\gamma(t))-t]+u(\gamma(0))$. Sending $t\rightarrow+\infty$, we obtain
$$
u(x)\leq\inf_{\gamma\in\mathscr{R}(u)}[u(\gamma(0))+b_{\gamma}(x)].
$$
On the other hand, by (i) of Proposition \ref{pro co-ray}, there is a $u$-ray $\gamma_{x}:[0,+\infty)\rightarrow X$ starting from $x\in X$. Thus,
$$
u(x)=u(\gamma_{x}(0))+b_{\gamma_{x}}(x)\geq\inf_{\gamma\in\mathscr{R}(u)}[u(\gamma(0))+b_{\gamma}(x)].
$$
This completes the proof of the proposition.
\end{proof}

\section{Set-assigned dl-functions and the induced pseudo-metric}
In this section, we prove that set-assigned dl-functions are well-defined and satisfy a simple, but useful observation: they attain the infimum within the set of all dl-functions that are nonnegative on the assigned set. Deduced from this observation, we obtain a direct but useful tool to study the large scale geometry, that is a pseudo-metric on the set of all compact subsets of $X$, i.e., $\mathcal{K}(X)$. Finally, we introduce the set of exposed points within a given compact subset.

\medskip
\subsection{Well-definedness of set-assigned dl-functions and their minimal characterizations}
The task of this part is to equip us with two elementary facts forming building blocks of this paper. First, we show the existence of set-assigned dl-functions, i.e., give a
\begin{proof}[Proof of Theorem \ref{sa sol}]
For any $x\in X$, we claim that $u^{r}_{K_0}(x)$ is
\begin{itemize}
  \item bounded from above,
  \item monotone increasing with respect to $r$ for $r\geq\sup_{x_0\in K_0}d(x,x_0)$.
\end{itemize}
Thus by sending $r\rightarrow+\infty$, $u_{K_0}(x):=\lim_{r\rightarrow+\infty}u^{r}_{K_0}(x)$ exists.

\medskip
First notice that
\begin{align*}
&u^{r}_{K_0}(x)=d(x,S\!_r(K_0))-r=d(x,x_r)-r\\
\leq&\,d(x, x_0)+d(x_r, x_0)-r=d(x,x_{0})\leq\sup_{x_0\in K_0} d(x,x_{0}),
\end{align*}
where $x_r\in S\!_r(K_0)$ is the foot of $x$ on $S\!_r(K_0)$ and $x_{0}$ is the foot of $x_r$ on $K_0$. This shows the first claim.

\medskip
On the other hand, for any $r_{2}\geq r_{1}\geq \sup_{x_0\in K_0}d(x,x_0)$, by compactness of $S\!_r(K_{0})$, there exist $y_{1}\in S_{r_{1}}(K_{0}), y_{2}\in S_{r_{2}}(K_{0})$ such that
\begin{equation}\label{eq-2.1}
d(x,y_{1})=d(x,S_{r_{1}}(K_0)),\quad d(x,y_{2})=d(x,S_{r_{2}}(K_0)).
\end{equation}
Continuity of $d(\cdot,K_0)$ implies that any minimal geodesic connecting $x$ with $y_{2}$ must intersect $S_{r_{1}}(K_{0})$ at some point $z_{1}$, see the figure below.

\begin{figure}[h]
   \caption{Sketched proof of Theorem \ref{sa sol}}
  \vspace{0.5cm}
  \small \centering
  \includegraphics[width=6cm,height=6cm]{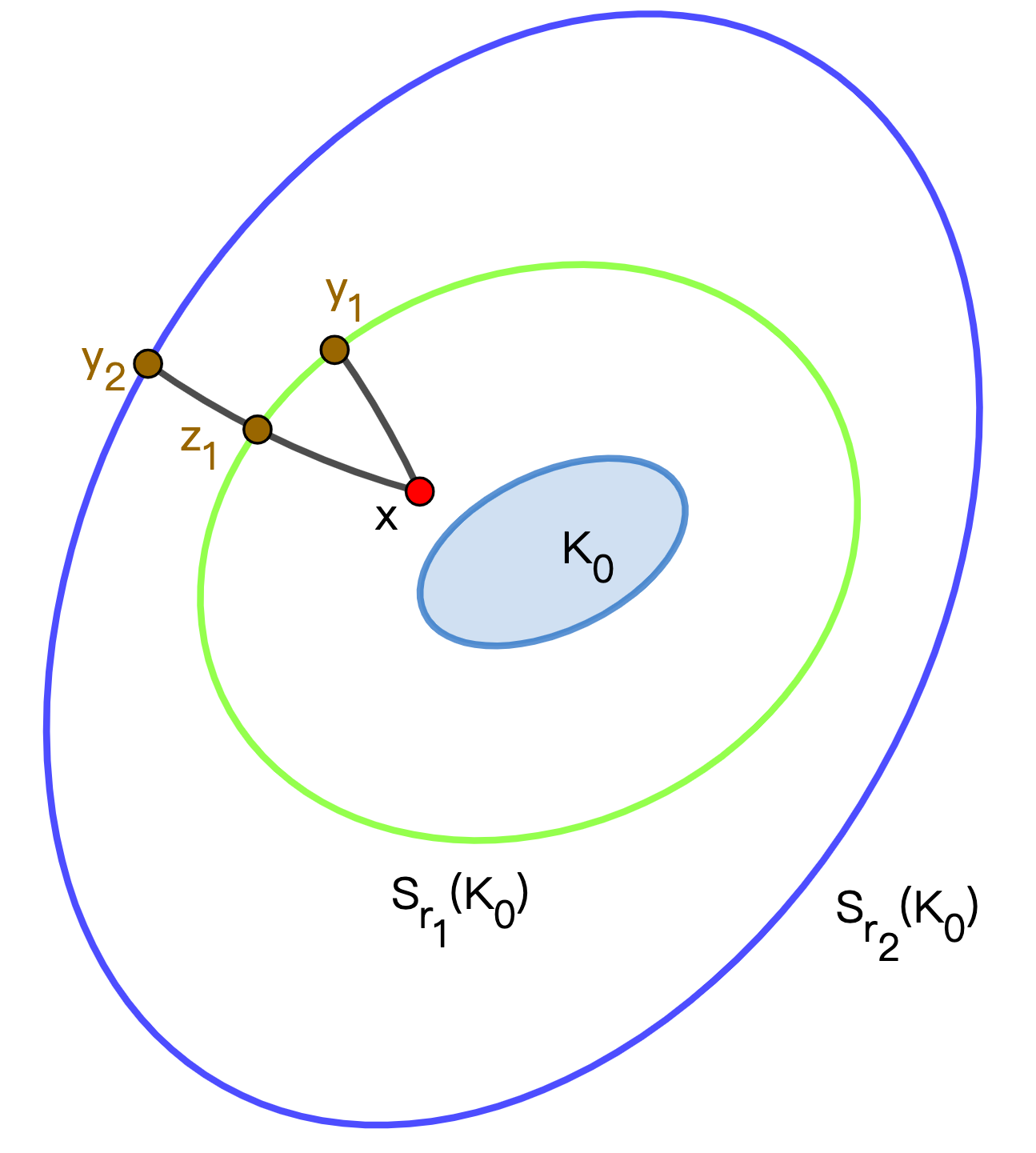}
  \label{picture of the proof}
\end{figure}


Therefore triangle inequality gives
\begin{equation}\label{eq-2.2}
d(x,y_{2})-d(x,z_{1})=d(z_{1},y_{2})\geq d(x_{0},y_{2})-d(x_{0},z_{1})=r_{2}-r_{1}.
\end{equation}
Combining \eqref{eq-2.2}, we obtain
\begin{equation*}
u^{r_{2}}_{K_{0}}(x)-u^{r_{1}}_{K_{0}}(x)=[d(x,y_{2})-d(x,y_{1})]-(r_{2}-r_{1})\geq d(x,z_{1})-d(x,y_{1})\geq 0,
\end{equation*}
where the last inequality follows from \eqref{eq-2.1}.
\end{proof}

From the definition of set-assigned dl-functions, we have
\begin{corollary}\label{sa base}
For any compact subset $K_{0}\subset X$,
\begin{enumerate}[(i)]
  \item $u_{K_0}|_{K_0}\geq0$ and $u_{K_0}|_{\text{Int}\,(K_0)}>0$,
  \item there is $x_0\in K_0$ such that $u_{K_0}(x_0)=0$, i.e., ep$(K_0)\neq\emptyset$.
\end{enumerate}
\end{corollary}

\begin{proof}
(i): For $r>0$ and $x_0\in K_{0}$, it is clear that $d(x_{0},S\!_r(K_0))\geq r$. Thus for every $r>0$, by definition,
\[
u^{r}_{K_0}(x_{0})\geq0.
\]
In particular, if $x_{0}\in\text{Int}\,(K_{0})$, there is $\delta>0$ such that $B_{\delta}(x_{0})\subset K_0$, then for any $r\geq\sup_{z_0\in K_0} d(x_0,z_0)$, there are $y\in S\!_r(K_0)$ and $z_0\in S_{\delta}(x_{0})$ such that
\begin{align*}
&u^{r}_{K_0}(x_{0})=d(x_{0},S\!_r(K_0))-r=d(x_{0},y)-r\\
=&\,\,d(x_{0},z_0)+[d(z_0,y)-r]\geq d(x_{0},z_0)+[d(K_0,y)-r]=\delta.
\end{align*}
Letting $r$ go to infinity, the conclusions of (i) hold.

\medskip
(ii): For every $n\geq1$, let $x_n$ be one of the foots of $S_n(K_0)$ on $K_0$, i.e., $d(x_n, S_n(K_0))=n$ and $x_n\in K_0$. Since $K_0$ is compact, by passing to a subsequence, we assume $x_n\rightarrow x_0\in K_0$. Then we have
\[
u_{K_0}(x_0)=\lim_{n\rightarrow\infty}[d(x_0, S_n(K_0))-n]=\lim_{n\rightarrow\infty}[d(x_{n}, S_n(K_0))-n]=0,
\]
or, equivalently, $x_0\in\,$ep\,$(K_0)$.
\end{proof}

As discussed in the introduction, the second elementary fact succeeds in characterizing set-assigned dl-functions among all dl-functions that are nonnegative on the base set. Now we turn into the

\begin{proof}[Proof of Theorem \ref{sa-Min}]
Let $u:X\rightarrow\mathbb{R}$ be a dl-function with $u|_{K_0}\geq0$. By (i) of Proposition \ref{pro co-ray}, for any $x\in X$, there exists a $u$-ray $\gamma_{x}:[0,+\infty)\rightarrow X$ starting from $x$ such that for $t\geq0$,
$$
u(x)-u(\gamma_{x}(t))=t=d(x,\gamma_{x}(t)).
$$
For $r>\sup_{x_0\in K_0}d(x,x_{0})$, we assume that $\gamma_{x}$ intersect $S\!_r(K_{0})$ at $\bar{y}$. Then by \eqref{r-sphere}, let $x_0\in K_0$ be the foot of $\bar{y}$ on $S\!_r(K_{0})$, it follows that
\begin{align*}
u(x)&\,=u(\bar{y})+d(x,\bar{y})\geq u(\bar{y})+d(x,S\!_r(K_{0}))\\
&\,\geq u(x_{0})-d(x_{0},\bar{y})+d(x,S\!_r(K_{0}))\\
&\,=u(x_{0})-r+d(x,S\!_r(K_{0}))\\
&\,\geq d(x,S\!_r(K_{0}))-r=u^{r}_{K_{0}}(x),
\end{align*}
where the second inequality holds since $u$ is 1-Lipschitz, see Remark \ref{1-Lip}, and the last inequality uses the fact that $u(x_{0})\geq0$. Now sending $r\rightarrow+\infty$, we could apply Theorem \ref{sa sol} to complete the proof.
\end{proof}

Combining Theorem \ref{sa-Min} and the above Corollary~\ref{sa base} gives
\begin{corollary}\label{K-u}
If $K_0\subset K_1\subset X$, then $u_{K_0}\leq u_{K_1}$.
\end{corollary}

\medskip
Recall that $\mathcal{K}(X)$ denotes the space of all compact subset of $X$ and let  $d_{H}(\cdot,\cdot)$ denote  the Hausdorff metric on $\mathcal{K}(X)$ induced by $d$, i.e., for $K_{0},K_{1}\in\mathcal{K}(X)$,
\[
d_{H}(K_{0},K_{1})=\max\{\max_{x\in K_1} d(x,K_0),\max_{x\in K_0} d(x,K_1)\}.
\]
As another application of Theorem \ref{sa-Min}, we shall show the following

\begin{prop}\label{hauss}
The map
$$
U:(\mathcal{K}(X),d_{H})\rightarrow(C(X,\mathbb{R}),d_{\infty}), \quad K_{0}\mapsto u_{K_{0}}
$$
is 1-Lipschitz, where $d_{\infty}$ is a metric on $C(X,\mathbb{R})$ defined by
\begin{equation}\label{d_infinity}
d_{\infty}(u,v):=\sup_{x\in X}|u(x)-v(x)|,
\end{equation}
for any $u,v\in C(X,\mathbb{R})$.
\end{prop}

\begin{proof}
For $K_{0},K_{1}\in\mathcal{K}(X)$, define $v:X\rightarrow\mathbb{R}$ as $x\mapsto u_{K_0}(x)+d_{H}(K_{0},K_{1})$. For every $x_{1}\in K_1$, by the definition of Hausdorff metric, there exists $x_{0}\in K_0$ such that $d(x_0,x_1)\leq d_{H}(K_{0},K_{1})$. Thus by $1$-Lipschitz property of dl-functions, we have
\[
v(x_1)\geq v(x_0)-d(x_0,x_1)=u_{K_0}(x_0)+[d_{H}(K_{0},K_{1})-d(x_0,x_1)]\geq0.
\]
Thus Theorem \ref{sa-Min} implies that for any $x\in X, v(x)\geq u_{K_1}(x)$ or
\[
u_{K_0}(x)-u_{K_1}(x)\geq-d_{H}(K_{0},K_{1}).
\]
Exchange the role of $K_0$ and $K_1$, we obtain $u_{K_1}(x)-u_{K_0}(x)\geq-d_{H}(K_{0},K_{1})$, and this completes the proof.
\end{proof}

\medskip
\subsection{A pseudo-metric on $\mathcal{K}(X)$ and the quotient metric space}
To establish Theorem \ref{pm}, we show an ``anti-triangle'' inequality among set-assigned dl-functions. For any two compact sets $K, K^{\prime}\subset X$, we shall use the notation
\[
u_{K}(K^{\prime}):=\min_{x\in K^{\prime}}\,\,u_{K}(x).
\]
\begin{lemma}\label{anti-triangle}
For any triple compact sets $K_0,K_1,K_2\in\mathcal{K}(X)$, we have
\begin{equation}\label{a-t}
u_{K_0}(K_1)+u_{K_1}(K_2)\leq u_{K_0}(K_2).
\end{equation}
\end{lemma}

\begin{proof}
Given any $K_0,K_1\subset X$, note that $v:X\rightarrow\mathbb{R}$ defined by
$$
v(z):=u_{K_0}(z)-u_{K_0}(K_1)
$$
is a dl-function which is nonnegative on $K_1$. By Theorem \ref{sa-Min}, for any $x\in X$,
\begin{equation*}
u_{K_1}(x)\leq v(x)=u_{K_0}(x)-u_{K_0}(K_1).
\end{equation*}
Thus we have
\[
u_{K_1}(K_2)=\min_{x\in K_2}u_{K_1}(x)\leq\min_{x\in K_2}[u_{K_0}(x)-u_{K_0}(K_1)]=u_{K_0}(K_2)-u_{K_0}(K_1).
\]
This completes the proof.
\end{proof}

Now we ready to give the
\medskip
\begin{proof}[Proof of Theorem \ref{pm}]
The symmetric property of $\rho:\mathcal{K}(X)\times\mathcal{K}(X)\rightarrow\mathbb{R}$ follows directly from the definition \eqref{pseudo metric}. Due to Lemma \ref{anti-triangle}, it is easy to see that for any triple $K_0,K_1,K_2\in\mathcal{K}(X)$,
\begin{align*}
&\,\rho(K_0,K_1)=-\frac{1}{2}[u_{K_0}(K_1)+u_{K_1}(K_0)]\geq-\frac{1}{2}u_{K_0}(K_0)=0,\\
&\,\rho(K_0,K_2)=-\frac{1}{2}[u_{K_0}(K_2)+u_{K_2}(K_0)]\\
\leq&\,-\frac{1}{2}[u_{K_0}(K_1)+u_{K_1}(K_2)+u_{K_2}(K_1)+u_{K_1}(K_0)]\\
=&\,-\frac{1}{2}[u_{K_0}(K_1)+u_{K_1}(K_0)]-\frac{1}{2}[u_{K_2}(K_1)+u_{K_1}(K_2)]\\
=&\,\rho(K_0,K_1)+\rho(K_1,K_2).
\end{align*}
By elementary knowledge of metric geometry, any pseudo metric defines an equivalence relation \cite[Proposition 1.1.5]{BBI}, i.e., $K_0\sim K_1$ if and only if $\rho(K_0,K_1)=0$. We introduce below an alternative formulation of this equivalence relation.

\medskip
\textit{Claim:} The following two assertions are equivalent:
\begin{enumerate}[(i)]
  \item  $K_0\sim K_1$;
  \item $\exists\,\, c\in \mathbb{R}$ such that $u_{K_0} = u_{K_1} + c$.
\end{enumerate}

\medskip
(i)$\Rightarrow$(ii): For any $x\in X$, choosing $K_2=\{x\}$ in \eqref{a-t}, we obtain that
\begin{equation}\label{ieq:1}
u_{K_0}(x)\geq u_{K_1}(x)+u_{K_0}(K_1).
\end{equation}
Since the position of $K_0, K_1$ in \eqref{ieq:1} is symmetric, we could exchange $K_0$ and $K_1$ to obtain that for any $x\in X$,
\[
u_{K_1}(x)\geq u_{K_0}(x)+u_{K_1}(K_0).
\]
But the assumption $K_0\sim K_1$ implies $u_{K_0}(K_1)=-u_{K_1}(K_0)$. So the above inequality reads as
\[
u_{K_0}(x)\leq u_{K_1}(x)+u_{K_0}(K_1),
\]
which completes the proof of theorem and shows that $c=u_{K_0}(K_1)$.

\medskip
(ii)$\Rightarrow$(i): Assume there is $c\in\mathbb{R}$ such that $u_{K_0},u_{K_1}$ satisfies
\begin{equation}\label{equi}
u_{K_0}(x)=u_{K_1}(x)+c,\quad\text{for any }x\in X.
\end{equation}
Since $u_{K}$ is non-negative on $K$,
\[
u_{K_0}(K_1)=\min_{x\in K_1}u_{K_0}(x)=\min_{x\in K_1}[u_{K_1}(x)+c]\geq c.
\]
On the other hand, by Corollary \ref{sa base}, choosing $x\in\text{ep}(K_1)$ in \eqref{equi}, it follows that
\[
u_{K_0}(K_1)\leq\min_{x\in\text{ep}K_1}u_{K_0}(x)=c\quad \text{and}\quad c=u_{K_0}(K_1).
\]
Again by definition of $\rho$ and \eqref{equi},
\[
0=\min_{x\in K_0}u_{K_0}(x)=\min_{x\in K_0}[u_{K_1}(x)+u_{K_0}(K_1)]=-2\rho(K_0,K_1).
\]
Hence $K_0\sim K_1$.
\end{proof}

\medskip
Furthermore, applying Proposition \ref{hauss}, the above facts lead  to the following immediate
\begin{corollary}\label{eqc-c}
For any $K_0,K_1\in\mathcal{K}(X)$,
\[
\rho(K_0,K_1)\leq d_{H}(K_0,K_1).
\]
In particular, the equivalence class of $\mathcal{K}(X)$ under $\sim$ are closed in $(\mathcal{K}(X),d_H)$.
\end{corollary}

From the point of view of metric geometry, it is natural to consider the quotient metric space $(\mathcal{K}(X)/\!\!\sim,d_{\text{{sa}}})$ induced by $\rho$. We use $[K_0]$ to denote the equivalence class of $K_0$ under $\sim$. Again following directly from Corollary \ref{eqc-c}, for any $K^{\prime}_0\in[K_0],K^{\prime}_1\in[K_1]$,
\[
d_{\text{sa}}([K_0],[K_1])\leq d_{H}(K^{\prime}_0,K^{\prime}_1).
\]
Thus every path on $(\mathcal{K}(X),d_H)$ reduces to a path on $(\mathcal{K}(X),d_{\text{sa}})$. Since $(X,d)$ is path-connected, so is $(\mathcal{K}(X),d_H)$ and the discussion implies the following

\begin{corollary}
$(\mathcal{K}(X)/\!\!\sim,d_{\text{sa}})$ is a path-connected metric space.
\end{corollary}

Fix $x_{0}\in X$, we identify $C(X,\mathbb{R})/\mathbb{R}$ with all continuous functions on $X$ vanishing at $x_{0}$, the metric $d_{\infty}$ given in \eqref{d_infinity} naturally acts on this subspace.
\begin{prop}
The map $U_{x_{0}}:(\mathcal{K}(X),\rho)\rightarrow(C(X,\mathbb{R})/\mathbb{R},d_{\infty})$,
\[
K_0\mapsto u_{K_0}(\cdot)-u_{K_0}(x_{0})
\]
is Lipschitz.
\end{prop}

\begin{proof}
By the definition of $d_{\infty}$, we have
\begin{align*}
&\,d_{\infty}(U_{x_{0}}(K_0),U_{x_{0}}(K_1))\\
=&\,\sup_{x\in X}|u_{K_0}(x)-u_{K_0}(x_{0})-u_{K_1}(x)+u_{K_1}(x_{0})|\\
\leq&\,\sup_{x\in X}|[u_{K_0}(x)-u_{K_1}(x)]+[u_{K_1}(x_{0})-u_{K_0}(x_{0})]|
\end{align*}
and by Theorem \ref{sa-Min}, for any $x_0,x\in X$, we get
\begin{equation*}
\begin{split}
u_{K_0}(K_1)\leq u_{K_0}(x)-u_{K_1}(x)\leq -u_{K_1}(K_0),\\
u_{K_1}(K_0)\leq u_{K_1}(x_{0})-u_{K_0}(x_{0})\leq -u_{K_0}(K_1).
\end{split}
\end{equation*}
We add the above two inequalities and take supremum in $x\in X$ to obtain
\[
-2\rho(K_0,K_1)\leq d_{\infty}(U_{x_{0}}(K_0),U_{x_{0}}(K_1))\leq 2\rho(K_0,K_1).
\]
This completes the proof.
\end{proof}
Moreover, we obtain
\begin{corollary}
The map $U_{x_{0}}$ induces a Lipschitz injection
\[
\bar{U}:(\mathcal{K}(X)/\!\!\sim,d_{\text{sa}})\rightarrow(C(X,\mathbb{R})/\mathbb{R},d_{\infty}).
\]
\end{corollary}


\medskip
\subsection{Exposed points of a given compact set}
As is pointed out in the introduction, for any $K_0\in\mathcal{K}(X)$, the set-assigned dl-function $u_{K_0}$ is determined   only by part of the assigned set. This fact is formulated via an important operator ep$:\mathcal{K}(X)\rightarrow\mathcal{K}(X)$. We are ready to give

\medskip
\begin{proof}[Proof of Theorem \ref{ep}]
(i):  The conclusion that ep$(K_0)\subseteq\partial K_0$ and is nonempty follows from Corollary~\ref{sa base}. Note that ep$(K_0)\subseteq K_0$, and thus by Corollary~\ref{K-u}, for any $x\in X$, we obtain
\[
u_{\text{ep}(K_0)}(x)\leq u_{K_0}(x).
\]
Combining the definition of ep$(K_0)$, we conclude that
\[
u_{\text{ep}(K_0)}|_{\text{ep}(K_0)}\equiv0.
\]
This is equivalent to the statement ep\,$\circ$\,ep\,$(K_0)$\,=\,ep\,$(K_0)$.

\medskip
(ii): By definition, $K_0\sim\text{ep}(K_0)$ is equivalent to the fact
\[
\rho(K_0, \text{ep}(K_0))=-\frac{1}{2}[u_{K_0}(\text{ep}(K_0))+u_{\text{ep}(K_0)}(K_0)]=-\frac{1}{2}u_{\text{ep}(K_0)}(K_0)=0,
\]
and thus we only need to show that $u_{\text{ep}(K_0)}(K_0)=0$.

\medskip
Assume on the contrary, there is $x_{\ast}\in K_0\setminus\text{ep}(K_0)$ such that $u_{\text{ep}(K_0)}(x_{\ast})<0$. Then there exist $\epsilon>0, r_0>0$ such that for $r\geq r_0$, since $x_{\ast}\in K_0$, we have
\begin{equation}\label{ieq:2}
\begin{split}
u^{r}_{K_0}(x_{\ast})&\,=d(x_{\ast},S\!_{r}(K_0))-r\geq-\epsilon,\\
u^{r}_{\text{ep}(K_0)}(x_{\ast})&\,=d(x_{\ast},S\!_{r}(\text{ep}(K_0)))-r<-2\epsilon.
\end{split}
\end{equation}
By definition of ep, there is $r_1>0$ such that for any $x\in S\!_{r}(K_0)$ with $r\geq r_1$,
\begin{equation}\label{ieq:3}
d(x,\text{ep}(K_0))-d(x,K_0)\leq\epsilon.
\end{equation}
We choose $r_2\geq\max\{r_0,r_1\}$ such that for $r\geq r_2$,
\[
d(K_0,S\!_{r}(\text{ep}(K_0)))\geq r_1.
\]
By \eqref{ieq:3}, for any $x^{\prime}\in S\!_{r}(\text{ep}(K_0))$, we obtain
\[
d(x^{\prime}, S\!_{r}(K_0))\leq\epsilon.
\]
Thus there is $x_r\in S\!_{r}(\text{ep}(K_0))$ such that
\begin{align*}
\,&u^{r}_{\text{ep}(K_0)}(x_{\ast})=d(x_{\ast},S\!_{r}(\text{ep}(K_0)))-r\\
=\,&d(x_{\ast},x_r)-r\geq\,d(x_{\ast},S\!_r(K_0))-r-\epsilon\\
\geq\,&u^{r}_{K_0}(x_{\ast})-\epsilon\geq-2\epsilon,
\end{align*}
which contradicts the second inequality of \eqref{ieq:2}.
\end{proof}

\begin{remark}\label{ep-2}
Heuristically, ep$(K_{0})$ consists of points of $K_0$ that are visible from infinity. In general ep$(K_0)\neq\partial K_0$. For instance, one can easily show that if $X$ is a Hadamard manifold (a complete, simply connected Riemannian manifold of nonpositive curvature), then ep$(K_0)=\,$co$K_0\cap\partial K_0$ for every compact $K_0\subset X$, where co$K$ denotes the smallest geodeiscally convex set containing $K$.
\end{remark}

According to Proposition \ref{pro co-ray}, for two non-empty compact sets $K_0, K$, we would like to introduce the following notation
\[
\Phi^{t}_{K_0}(K)=\{\gamma(t)\in X\,|\,\gamma:[0,t]\rightarrow X\text{ is a segment of co-ray of }u_{K_0} \text{ with }\gamma(0)\in K\}.
\]
Notice that fixing $K_0\in\mathcal{K}(X),  \Phi^{\cdot}_{K_0}:[0,+\infty)\times\mathcal{K}(X)\rightarrow\mathcal{K}(X)$ defines a semiflow on $\mathcal{K}(X)$, i.e., for any $t,s\geq0$ and $K\in\mathcal{K}(X)$,
\[
\Phi^{t+s}_{K_0}(K)=\Phi^{t}_{K_0}\circ\Phi^{s}_{K_0}(K).
\]
The equivalence between ep$(K_0)$ and its image under $\Phi^{t}_{K_0}$ is implied by
\begin{prop}\label{equi-2}
For every $t\geq0$,
\[
u_{\Phi^{t}_{K_0}(\text{ep}(K_0))}(x)=u_{\text{ep}(K_0)}(x)+t\quad \text{for any}\quad x\in X.
\]
In particular, $\Phi^{t}_{K_0}(\text{ep}(K_0))\sim\text{ep}(K_0)$.
\end{prop}

\begin{proof}
By (i) of Theorem \ref{ep}, $u_{\text{ep}(K_0)}|_{\text{ep}(K_0)}\equiv0$. Proposition \ref{gradient} gives
\[
u_{\text{ep}(K_0)}(x)\equiv-t,\quad\text{for}\quad x\in\Phi^{t}_{K_0}(\text{ep}(K_0)).
\]
Thus $u_{\text{ep}(K_0)}(x)+t$ is nonnegative on $\Phi^{t}_{K_0}(\text{ep}(K_0))$ and by Theorem \ref{sa-Min},
\begin{equation}\label{ieq:4}
u_{\Phi^{t}_{K_0}(\text{ep}(K_0))}(x)-t\leq u_{\text{ep}(K_0)}(x),\quad\text{for any}\quad x\in X.
\end{equation}

\medskip
For the other direction, i.e.,
\[
u_{\text{ep}(K_0)}(x)\leq u_{\Phi^{t}_{K_0}(\text{ep}(K_0))}(x)-t,
\]
again by Theorem \ref{sa-Min}, it is sufficient to show that
\[
u_{\Phi^{t}_{K_0}(\text{ep}(K_0))}(x)-t\geq0\quad\text{for any }x\in\text{ep}(K_0).
\]
We shall argue by contradiction. Assume there exist $\epsilon>0$ and $x_{\ast}\in\,\,$ep$(K_0)$ such that
\[
u_{\Phi^{t}_{K_0}(\text{ep}(K_0))}(x_{\ast})-t<-2\epsilon.
\]
Thus there is a diverging sequence $\{x_i\}_{i\in\mathbb{N}}$ with
\begin{align*}
&d(x_i,\text{ep}(K_0))-d(x_i,\Phi^{t}_{K_0}(\text{ep}(K_0)))-t\\
\leq&\,[d(x_i,x_{\ast})-d(x_i,\Phi^{t}_{K_0}(\text{ep}(K_0)))]-t\\
=&\,[d(S_{r_i}(\Phi^{t}_{K_0}(\text{ep}(K_0))),x_{\ast})-r_i]-t\\
=&\,u^{r_i}_{\Phi^{t}_{K_0}(\text{ep}(K_0))}(x_{\ast})-t<-\epsilon,
\end{align*}
where $x_i\in S_{r_i}(\Phi^{t}_{K_0}(\text{ep}(K_0)))$ such that
\[
d(x_i, x_{\ast})=d(S_{r_i}(\Phi^{t}_{K_0}(\text{ep}(K_0))),x_{\ast})
\]
and $r_i$ diverges to infinity.

\medskip
Let $\gamma_i$ be a minimal geodesic connecting ep$(K_0)$ with $x_i$ and $\gamma_i(0)\in$\,ep$(K_0)$. By the compactness of ep$(K_0)$, $\gamma_i$ converges uniformly to some co-ray $\gamma:[0,+\infty)\rightarrow X$ of $u_{\text{ep}(K_0)}$ on compact intervals. Then $\gamma_i(t)\rightarrow\gamma(t)\in\Phi^{t}_{K_0}(\text{ep}(K_0))$, thus for $i$ sufficiently large,
\begin{align*}
&\,0=d(x_i,\gamma_i(0))-d(x_i,\gamma_i(t))-t\\
\leq&\,[d(x_i,\text{ep}(K_0))-d(x_i,\Phi^{t}_{K_0}(\text{ep}(K_0)))-t]+\epsilon<0.
\end{align*}
This contradiction completes the proof.
\end{proof}

\section{Gromov-Hausdorff metric and set-assigned dl-functions}
It is well-known that the Gromov-Hausdorff metric is defined on the space of metric spaces. In this section, we shall present a result showing that set-assigned dl-functions are continuous with respect to the underlying set-assigned geodesic spaces in the sense of Gromov-Hausdorff metric. For convenience, we use the notation
\[
d_{GH}:=d_{GH}((X_0, d_0, K_0),(X_1, d_1, K_1))
\]
if there is no confusion. To go further, we introduce the notion of correspondence between two set-assigned metric spaces and its distortion. Note that our definition is a modification of the standard one \cite[Definition 7.3.17, 7.3.21]{BBI} without concerning the assigned set.
\begin{definition}\label{cor-dis}
Let $(X_0, d_0, K_0)$ and $(X_1, d_1, K_1)$ be two set-assigned metric spaces, we call $\mathfrak{R}\subseteq X_0\times X_1$ a set-assigned correspondence if
\begin{itemize}
  \item for any $x_0\in K_0$, there exists $x_1\in K_1$ such that $(x_0,x_1)\in\mathfrak{R}$,
  \item for any $x_1\in K_1$, there exists $x_0\in K_0$ such that $(x_0,x_1)\in\mathfrak{R}$,
  \item for any $x_0\in X_0$, there exists $x_1\in X_1$ such that $(x_0,x_1)\in\mathfrak{R}$,
  \item for any $x_1\in X_1$, there exists $x_0\in X_0$ such that $(x_0,x_1)\in\mathfrak{R}$.
\end{itemize}
We associate every correspondence with a positive number, called distortion, by
\begin{equation}
\text{dis}\,\,\mathfrak{R}:=\sup\{\,|d_0(x_0,x'_0)-d_1(x_1,x'_1)|\, | \,(x_0,x_1),(x'_0,x'_1)\in\mathfrak{R}\}.
\end{equation}
\end{definition}

The distortion of a correspondence is equivalent to the Gromov-Hausdorff metric between the underlying two metric spaces.
\begin{prop}\label{d-dis}
\begin{equation}\label{GH-dis}
d_{GH}\leq\inf_{\mathfrak{R}\subseteq X_0\times X_1}\text{dis}\,\,\mathfrak{R}\leq2d_{GH}.
\end{equation}
\end{prop}

\begin{proof}
For any $r>d_{GH}$, assume that $X_0,X_1$ are subspaces of some metric space $Z$ such that
\begin{equation}\label{assume}
d_{H}(X_0,X_1)+d_H(K_0,K_1)<r.
\end{equation}
We define
\begin{equation*}
\mathfrak{R}=\{(x_0,x_1)\in X_0\times X_1\,|\,d(x_0,x_1)<r\}.
\end{equation*}
Then by \eqref{assume}, we obtain that $\mathfrak{R}$ is a set-assigned correspondence. Since $d_{H}(X_0, X_1)<r$, if $(x_0,x_1),(x'_0,x'_1)\in\mathfrak{R}$,
\begin{equation*}
|d_0(x_0,x'_0)-d_1(x_1,x'_1)|\leq d(x_0,x_1)+d(x'_0,x'_1)<2r.
\end{equation*}

\vspace{1em}
Now we prove the other inequality. Assume $\text{ dis}\,\,\mathfrak{R}=2r$ for some correspondence. Given $x_0\in X_0,x_1\in X_1$, we define
\begin{equation*}
d(x_0,x_1)=\inf\,\,\{d_0(x_0,x'_0)+d_1(x_1,x'_1)+r\,|\,(x'_0,x'_1)\in\mathfrak{R}\}.
\end{equation*}
Then we have
\begin{itemize}
  \item $d$ satisfies triangle inequality and in fact, $d$ is metric on $X_0\sqcup X_1$,
  \item $d_{H}(X_0,X_1)\leq r, \quad d_{H}(K_0,K_1)\leq r$.
\end{itemize}
So we have $d_{GH}\leq2r=\text{ dis }\mathfrak{R}$.
\end{proof}

It is reasonable to believe that the existence of $\epsilon$-isometry is closely related to the smallness of Gromov-Hausdorff metric.
\begin{prop}\label{e-isometry}
Let $(X_0, d_0, K_0),(X_1, d_1, K_1)$ be two set-assigned metric spaces,
\begin{enumerate}
  \item if $d_{GH}<\epsilon$, then there exists a $2\epsilon$-isometry $f:(X_0,K_0)\rightarrow(X_1,K_1)$,
  \item if there exists an $\epsilon$-isometry $f:(X_0,K_{0})\rightarrow(X_1,K_{1})$, then $d_{GH}\leq3\epsilon$.
\end{enumerate}
\end{prop}

\begin{proof}
If $d_{GH}<\epsilon$, by Proposition \ref{d-dis}, there is a correspondence, say $\mathfrak{R}$, such that dis$\,\mathfrak{R}<2\epsilon$. Define $f:(X_0,K_0)\rightarrow(X_1,K_1)$ as: for $x_0\in X_0$ (in particular $x_0\in K_0$), choose $f(x_0)\in X_1$ (in particular $f(x_0)\in K_1$) with $(x_0,f(x_0))\in\mathfrak{R}$. It is clear that dis $f\leq$ dis$\,\mathfrak{R}<2\epsilon$. For $x_1\in X_1$, consider an $x_0\in X_0$ such that $(x_0,x_1)\in\mathfrak{R}$, both $x_1$ and $f(x_0)$ are in correspondence with $x_0$, then
\begin{equation}
d(x_1,f(x_0))\leq \text{ dis }\mathfrak{R}<2\epsilon,
\end{equation}
i.e., $f(X_0)$ is a $2\epsilon$-net in $X_1$. Similarly, $f(K_0)\subset K_1$ and is a $2\epsilon$-net of $K_1$. This completes the proof of (1).

\vspace{1em}
Let $f:(X_0,K_0)\rightarrow(X_1,K_1)$ be an $\epsilon$-isometry. Define
$$
\mathfrak{R}:=\{(x_0,x_1)\in X_0\times X_1: d_1(f(x_0),x_1)\leq\epsilon\},
$$
then by Definition \ref{cor-dis}, \,\,$\mathfrak{R}$ is a set-assigned correspondence. If $(x,y),(x',y')$ are two elements in $\mathfrak{R}$, then
\begin{align*}
&\,|d(y,y')-d(x,x')|\\
\leq&\,|d(f(x),f(x'))-d(x,x')|+d(y,f(x))+d(y',f(x'))\\
\leq&\,\text{ dis }f+2\epsilon\leq3\epsilon.
\end{align*}
Hence dis$\,\mathfrak{R}\leq3\epsilon$ so that $d_{GH}\leq3\epsilon$.
\end{proof}

Equipping with the above propositions, let us give
\begin{proof}[Proof of Theorem \ref{sa-GH}]
The existence of $2\epsilon$-isometry $f:(X_0,K_0)\rightarrow(X_1,K_1)$ is guaranteed by Proposition \ref{e-isometry}. By using the notations
\begin{equation*}
\begin{split}
u^{r}_{K_{0}}(x_0):=d_0(x_0,S\!_r(K_{0}))-r,\\
u^{r}_{K_{1}}(x_1):=d_1(x_1,S\!_r(K_{1}))-r,
\end{split}
\end{equation*}
defined in Section 3, there exist $x^{\ast}_1\in S\!_r(K_{1}),x^{\ast}_0\in X_0$ such that
\begin{equation*}
u_{K_{1}}^{r}(f(x_0))=d_1(f(x_0),x^{\ast}_1)-r\geq d_1(f(x_0),f(x^{\ast}_0))-r-2\epsilon,
\end{equation*}
where the existence of $x^{\ast}_0$ follows from the fact that $f(X)$ is a $2\epsilon$-net in $Y$ and it follows that
\begin{equation}\label{r'}
r^{\ast}:=d_{0}(x^{\ast}_0,K_{0})\in[r-6\epsilon, r+6\epsilon].
\end{equation}
Using the notation and estimate in \eqref{r'}, we obtain
\begin{align*}
&\,u^{r^{\ast}}_{K_{0}}(x_0)-u^{r}_{K_{1}}(f(x_0))=[d_0(x_0,S_{r^{\ast}}(K_{0}))-r^{\ast}]-u_{K_{1}}^{r}(f(x_0))\\
\leq&\,[d_0(x_0,x^{\ast}_0)-r^{\ast}]-[d_1(f(x_0),f(x^{\ast}_0))-r-2\epsilon]\\
=&\,[d_0(x_0,x^{\ast}_0)-d_1(f(x_0),f(x^{\ast}_0))]+[r-r^{\ast}]+2\epsilon\\
<&\,2\epsilon+6\epsilon+2\epsilon=10\epsilon.
\end{align*}
By sending $r\rightarrow\infty$ (so that $r^{\ast}\rightarrow\infty$ as well), $u_{K_{0}}(x_0)-u_{K_{1}}(f(x_0))<10\epsilon$.

\vspace{0.4cm}
On the other hand, there exists $x'_0\in S\!_r(K_{0})$ such that $d_0(x_0, S\!_r(K_{0}))=d_0(x_0,x'_0)$. Thus by setting $r':=d_1(f(x'_0),K_1)$, it is clear that $r'\in[r-4\epsilon,r+4\epsilon]$  and so
\begin{align*}
&\,u^{r'}_{K_{1}}(f(x_0))-u^{r}_{K_{0}}(x_0)=u_{K_{1}}^{r'}(f(x_0))-[d_0(x_0,x'_0)-r]\\
\leq&\,[d_1(f(x_0),f(x'_0))-r']-[d_0(x_0,x'_0)-r]\\
=&\,[d_1(f(x_0),f(x'_0))-d_0(x_0,x'_0)]+[r-r']\\
<&\,4\epsilon+2\epsilon=6\epsilon.
\end{align*}
By sending $r\rightarrow\infty$ (so that $r^{\ast}\rightarrow\infty$ as well), $u_{K_{1}}(f(x_0))-u_{K_{0}}(x_0)<6\epsilon$. This completes the proof.
\end{proof}

\section{Further results when $K$ is a singleton}
In this section, we consider the special case when assigned compact set $K_0$ is a singleton $\{x_0\}$. By regarding $X$ as a subset of $\mathcal{K}(X)$, we give further results concerning the metric geometry of $X$ via the restricted pseudo-metric $\rho|_{X\times X}$. More precisely, we use the simplified notion $u_{x_0}$ instead of $u_{\{x_0\}}$ and consider the subclass
\[
\{u_{x_0}\,|\,x_0\in X\}
\]
of set-assigned dl-functions. With a little abuse of notations, we define $\rho:X\times X\rightarrow\mathbb{R}$ by
\begin{equation}\label{pt-pm}
\rho(x_0,x_1)=-\frac{1}{2}[u_{x_0}(x_1)+u_{x_1}(x_0)].
\end{equation}
Applying Theorem \ref{pm} to singletons, we deduce that $\rho$ is a pseudo-metric on $X$. And this pseudo metric is just the restriction of the pseudo-metric $\rho$ defined by \eqref{pseudo metric} on $X\times X$. Now the non-negativeness of $\rho$ and Corollary \ref{eqc-c} read as for any $x_0,x_1\in X$,
\begin{equation}\label{eqc-c1}
0\leq \rho(x_0,x_1)\leq d(x_0,x_1).
\end{equation}

\subsection{Two extreme cases.}
The inequality \eqref{eqc-c1} leads to two extreme situations, and we can deduce geometric properties of $X$ from these information.

\vspace{1em}
Assume $x_{0}\in X$ satisfies that for any $x\in X$, there is a geodesic ray starting from $x_0$ and passing through $x$, then
\[
u_{x_{0}}(x)=-d(x_{0},x),\quad\text{for any }x\in X.
\]

\begin{example}
If $X=(M,g)$ is a complete smooth Riemannian manifold, a point $x_0\in X$ possessing the property described above is called a pole \cite[Page 151, Remark 3.4]{Carmo}.
\end{example}

In addition, if a pair of points $x_0, x_1\in X$ lies on a geodesic line, then it follows from Definition \ref{sa sol} and \eqref{pt-pm} that
\[
u_{x_0}(x_1)=u_{x_1}(x_0)=-d(x_0,x_1)\quad\text{and}\quad\rho(x_0,x_1)=d(x_0,x_1).
\]
From the above discussion, we obtain

\begin{prop}
If $\gamma:\mathbb{R}\rightarrow X$ is a geodesic line, then the inclusion map $i:(\text{Im}(\gamma), d)\rightarrow(X,\rho)$ is an isometry.
\end{prop}

\begin{example}
If $X$ is a Hadamard manifold, i.e., a simply connected, complete Riemannian manifold of non-positive curvature, then every complete geodesic is a geodesic line. Thus every pair of distinct points $x_0, x_1$ lies on a geodesic line and the pseudo-metric $\rho$ on $X$ coincides with the metric induced by $g$.
\end{example}

The following theorem deals with the first case when $\rho$ coincides with the original metric $d$ everywhere.
\begin{theorem}\label{pt-1}
For any $x_0,x_1\in X$, the pseudo-metric $\rho:X\times X\rightarrow\mathbb{R}$ satisfies
\begin{equation}\label{eq-pm}
\rho(x_0,x_1)=d(x_0,x_1),
\end{equation}
if and only if for every pair of distinct points $x_0, x_1$, there is a geodesic line passing through them.
\end{theorem}

\begin{proof}
Notice that for any $x_0\in X, u_{x_0}(x_0)=0$. Now Proposition \ref{hauss} states that, for any $x_0,x_1\in X$,
\[
|u_{x_0}(x_1)|=|u_{x_0}(x_1)-u_{x_0}(x_0)|\leq d(x_0,x_1),\quad |u_{x_1}(x_0)|\leq d(x_0,x_1).
\]
Combining the above inequalities and the assumption \eqref{eq-pm}, we have for any $x_0,x_1\in X$,
\begin{equation}\label{pt-3}
u_{x_0}(x_1)=u_{x_1}(x_0)=-d(x_0,x_1).
\end{equation}

For any pair of distinct points $x_0,x_1\in X$, we set $a=d(x_0,x_1)>0$. We choose a unit speed minimal geodesic segment $\xi:[0,a]\rightarrow X$ with $\xi(0)=x_0,\xi(a)=x_1$. By Proposition \ref{pro co-ray}, there is a unit speed co-ray $\eta:[0,+\infty)\rightarrow X$ of $u_{x_0}$ starting from $x_1$. We construct a unit speed curve $\gamma:[0,+\infty)\rightarrow X$ by
\begin{equation}\label{gl-1}
\gamma(t)=
\begin{cases}
\xi(t),\quad t\in[0,a],\\
\eta(t-a),\quad t\in[a,\infty).
\end{cases}
\end{equation}
\textit{Claim:} $\gamma$ is a geodesic ray.

\vspace{1em}
\textit{Proof of the claim:} By \eqref{pt-3} and definition \eqref{gl-1}, for $t\in[0,a]$,
\begin{equation}\label{pt-4}
u_{x_0}(\gamma(t))=u_{x_0}(\xi(t))=-d(x_0,\xi(t))=-d(\xi(0),\xi(t))=-t.
\end{equation}
For $t\geq a$, by definition \eqref{gl-1} and Proposition \ref{gradient},
\begin{align*}
&\,u_{x_0}(\gamma(t))=u_{x_0}(\eta(t-a))\\
=&\,u_{x_0}(\eta(0))-(t-a)\\
=&\,-d(x_0,\eta(0))-(t-a)\\
=&\,-d(x_0,x_1)-(t-a)=-t.
\end{align*}
Combining \eqref{pt-4} and above equalities, we have for any $0\leq t_1\leq t_2$,
\begin{equation}\label{pt-5}
u_{x_0}(\gamma(t_2))-u_{x_0}(\gamma(t_1))=t_1-t_2.
\end{equation}
Thus Remark \ref{ray} implies that $\gamma$ is a unit speed geodesic ray.

\vspace{1em}
Now for any $n\in\mathbb{N}$, set $x_n=\gamma(na)$. Again by Proposition \ref{pro co-ray}, there is a unit speed co-ray $\eta_n:[0,+\infty)\rightarrow X$ of $u_{x_n}$ starting from $x_0$. Now we construct a sequence of unit speed curve $\gamma_n:(-\infty,na]\rightarrow X$ by
\begin{equation*}
\gamma_n(t)=
\begin{cases}
\eta_n(-t),\quad t\in(-\infty,0],\\
\gamma(t),\quad t\in[0,na].
\end{cases}
\end{equation*}
Since $\gamma|_{[0,na]}$ is a unit speed minimal geodesic segment, with $x_0, x_1$ replaced by $x_n, x_0$ respectively, the above claim shows that $\gamma_n$ is a unit speed geodesic ray. Using the local compactness of $X$ and Ascoli-Arzel\`{a} theorem, up to a subsequence, $\gamma_n:(-\infty,na]\rightarrow X$ converges uniformly on compact intervals to a 1-Lipschitz curve $\bar{\gamma}:\mathbb{R}\rightarrow X$ with $\bar{\gamma}|_{[0,\infty)}=\gamma$. Since every $\gamma_n$ is a unit speed geodesic ray, by the lower semi-continuity of length functional and the definition of $\bar{\gamma}$, for any $t_1\leq t_2$,
\begin{align*}
&\,d(\bar{\gamma}(t_1),\bar{\gamma}(t_2))=\lim_{n\rightarrow\infty}d(\gamma_n(t_1),\gamma_n(t_2))\\
=&\,t_2-t_1=\liminf_{n\rightarrow\infty}L(\gamma_n|_{[t_1,t_2]})\geq L(\bar{\gamma}|_{[t_1,t_2]}).
\end{align*}
Thus $\bar{\gamma}$ is a geodesic line with
\[
\bar{\gamma}(0)=\gamma(0)=x_0,\quad\bar{\gamma}(a)=\gamma(a)=x_1,
\]
which completes the proof.
\end{proof}

\begin{remark}
If we only assume $\rho=d$ for some pair of distinct points $x_0,x_1\in X, a=d(x_0,x_1)>0$, then there are two unit speed geodesic rays $\gamma^{x_1}_{x_0},\gamma^{x_0}_{x_1}:[0,+\infty)\rightarrow X$ such that
\begin{equation}
\begin{split}
\gamma^{x_1}_{x_0}(0)=x_0,\quad \gamma^{x_1}_{x_0}(a)=x_1,\\
\gamma^{x_0}_{x_1}(0)=x_1,\quad \gamma^{x_0}_{x_1}(a)=x_0.
\end{split}
\end{equation}
Notice that $\gamma^{x_1}_{x_0},\gamma^{x_0}_{x_1}$ may not be subrays of a same geodesic line.
\end{remark}

\begin{corollary}\label{pt-nb}
Assume the pseudo-metric $\rho:X\times X\rightarrow\mathbb{R}$ coincides with the original metric $d$ and in addition, $X$ satisfies the non-branching property (NB), then
\begin{enumerate}
  \item for any pair of distinct points $x_0, x_1\in X$ and $a=d(x_0, x_1)>0$, the unit speed geodesic line $\gamma^{x_1}_{x_0}:\mathbb{R}\rightarrow X$ with $\gamma^{x_1}_{x_0}(0)=x_0, \gamma^{x_1}_{x_0}(a)=x_1$ is unique.
  \item every unit speed complete geodesic $\gamma:\mathbb{R}\rightarrow X$ is a geodesic line.
\end{enumerate}
\end{corollary}

\begin{proof}
The first conclusion is a direct consequence of (NB). We now turn to the second one: by definition of geodesics, there are $t_0<t_1<t_2$ such that $\gamma|_{[t_0,t_2]}$ is a minimal geodesic segment and $x_0=\gamma(t_0),x_1=\gamma(t_1)$ are distinct points. By (1), there is a unique unit speed geodesic line $\gamma^{x_1}_{x_0}:\mathbb{R}\rightarrow X$ with $\gamma^{x_1}_{x_0}(t_0)=x_0$ and $\gamma^{x_1}_{x_0}(t_1)=x_1$. Then (NB) implies that for $t\in[t_0,t_2], \gamma(t)=\gamma^{x_1}_{x_0}(t)$. Set
\begin{equation}\label{pt-6}
\begin{split}
T_+=\sup\,\,\{\bar{t}\in\mathbb{R}\,|\,\,\gamma(t)=\gamma^{x_1}_{x_0}(t)\text{ for any }t\in[t_0,\bar{t}]\},\\
T_-=\inf\,\,\{\underline{t}\in\mathbb{R}\,|\,\,\gamma(t)=\gamma^{x_1}_{x_0}(t)\text{ for any }t\in[\underline{t},t_2]\}.
\end{split}
\end{equation}
Notice that $T_-\leq t_0, T_+\geq t_1$. We prove that $T_+=+\infty$ by contradiction: assume $T_+<+\infty$, the continuity of $\gamma$ and $\gamma^{x_1}_{x_0}$ implies that $\gamma(T_+)=\gamma^{x_1}_{x_0}(T_+)$. Since $\gamma$ is a geodesic, there is $\delta>0$ such that $\gamma|_{[T_{+}-\delta,T_{+}+\delta]}$ is minimal. Thus the two unit speed minimal geodesic segments $\gamma|_{[T_{+}-\delta,T_{+}+\delta]}$ and $\gamma^{x_1}_{x_0}|_{[T_{+}-\delta,T_{+}+\delta]}$ satisfy $\gamma(T_{+}-\delta)=\gamma^{x_1}_{x_0}(T_{+}-\delta)$ and $\gamma(T_{+})=\gamma^{x_1}_{x_0}(T_{+})$. It follows from (NB) that $\gamma(t)=\gamma^{x_1}_{x_0}(t)$ for any $t\in[t_0,T_{+}+\delta]$, this contradicts \eqref{pt-6}. In a similar fashion, $T_-=-\infty$. So we conclude that $\gamma=\gamma^{x_1}_{x_0}$ is a geodesic line.
\end{proof}

\begin{remark}
If (NB) is not assumed for $X$, then the first conclusion of the above corollary fails, see Figure~\ref{NBeg} below. But we do not have an example for which the second one fails.
\begin{figure}[h]
\caption{The example without (NB) assumption}\label{NBeg}
  \small \centering
  \includegraphics[width=8cm,height=3cm]{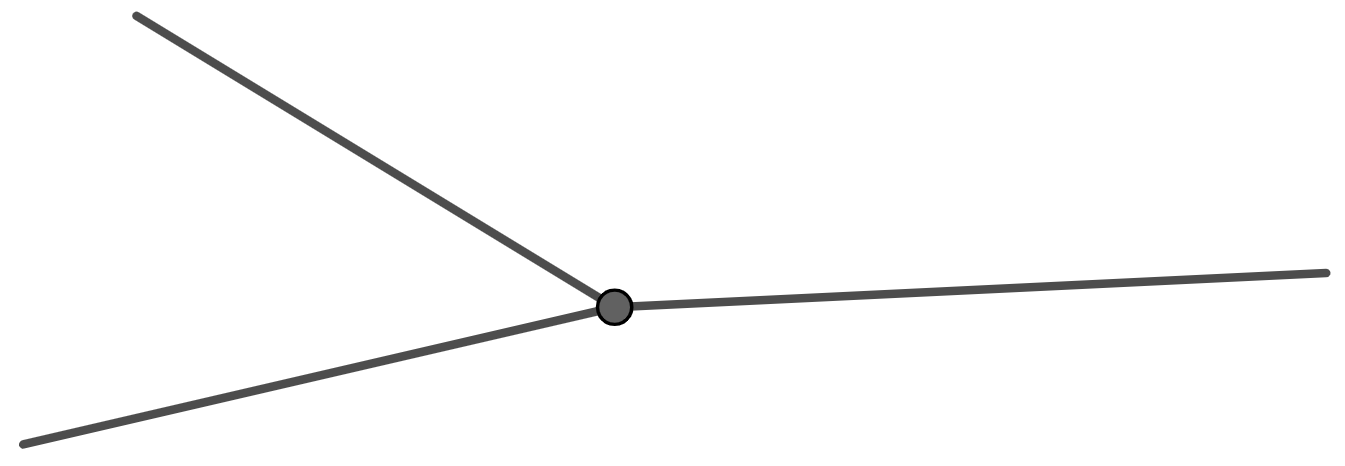}
\end{figure}
\end{remark}

\begin{remark}
In particular, if $X$ is an $n$-dimensional complete, connected Riemannian manifold, then for every $x\in X$, the exponential map $\exp_{x}:T_{x}X\rightarrow X$ is a diffeomorphism. Thus $X$ is diffeomorphic to $\mathbb{R}^n$.
\end{remark}

\vspace{0.3cm}
The second case relates the degeneracy of pseudo-metric $\rho$ to the structure of the set of dl-functions.
\begin{theorem}\label{pt-2}
Assume the pseudo-metric $\rho:X\times X\rightarrow\mathbb{R}$ vanishes everywhere, i.e., $\rho(x_0,x_1)=0$ for any $x_0,x_1\in X$, then, up to constants, there is only one dl-function.
\end{theorem}

The proof of Theorem \ref{pt-2} relies on a direct application of Theorem \ref{sa-Min}, which can be seen as another representation of dl-functions.
\begin{prop}\label{prop:rep2}
For any dl-function $u:X\rightarrow\mathbb{R}$ and $x\in X$,
\begin{equation}\label{rep2}
u(x)=\sup_{y\in X}\,\,[u(y)+u_{y}(x)].
\end{equation}
\end{prop}

\begin{proof}
For any $y\in X$, the function $v:X\rightarrow\mathbb{R}$ defined by
\[
v(x)=u(x)-u(y)
\]
is a dl-function (since $u$ is a dl-function) vanishing at $y$. By Theorem \ref{sa-Min},
\[
u_y(x)\leq v(x),
\]
which is equivalent to
\[
u(x)\geq u_y(x)+u(y) .
\]
Since the left hand side of the above inequality is independent of $y$, we take supremum for $y$ on both sides to obtain
\[
u(x)\geq\sup_{y\in X}\,\,[u(y)+u_{y}(x)].
\]
On the other hand, it is easily seen that the supremum on the right hand side is attained at $y=x$ with value $u(x)$.
\end{proof}

\begin{remark}
Assume $X$ is a smooth Riemannian manifold, then by Remark \ref{vis}, any dl-function is locally semiconcave. An interesting aspect of the above formula is that a semiconcave function can be obtained as the \textbf{supremum}, instead of infimum, of a family of semiconcave function. For example, we can compare Proposition~\ref{prop:rep2} with Proposition~\ref{representation formula} in Section~\ref{sec:representation formula}, we believe that there is a \textbf{beautiful duality} between Busemann function and set-assigned dl-functions we defined in this paper.
\end{remark}

\begin{proof}[Proof of Theorem \ref{pt-2}]
Fix $y_0\in X$. The assumption $\rho\equiv0$ and Theorem \ref{pm} indicate that
\[
u_y(x)=u_{y_0}(x)+c(y),\quad\text{ for all }x, y\in X.
\]
Then by \eqref{rep2} and above equality,
\begin{align*}
&\,u(x)=\sup_{y\in X}\,\,[u(y)+u_{y}(x)]\\
=&\,\sup_{y\in X}\,\,[u(y)+u_{y_0}(x)+c(y)]\\
=&\,u_{y_0}(x)+\sup_{y\in X}\,\,[u(y)+c(y)]\\
=&\,u_{y_0}(x)+\text{ const }.
\end{align*}
Thus, up to constants, there is only one dl-function on $X$, namely $u_{y_0}$.
\end{proof}

We want to explore the underlying geometry of $X$ when the pseudo-metric is degenerate. Unfortunately, we only obtain partial results (when $X$ is a smooth Riemannian manifold) on this topic. In the following, $\nabla^{+}u(x)$ denotes the set of super-gradients of $u$ at $x$ with respect to $g$ and we use the Riemannian volume as a natural Borel regular measure on $X$. The result can be formulated as

\begin{theorem}\label{uni-criterion}
Let $X$ be a connected, complete, non-compact, smooth \textbf{Riemannian manifold}. If $X$ admits only one dl-function, then for every point $x\in X$ except those in a set of measure zero, there is only one unit speed geodesic ray starting from $x$.
\end{theorem}

\begin{proof}
Due to Remarks~\ref{eikonal} and \ref{vis}, a dl-function $u:X\rightarrow\mathbb{R}$ is a viscosity solution to the eikonal equation \eqref{eik}. Proposition \ref{pro co-ray} states that for any point $x\in X$, there is at least one $u$-ray $\gamma:[0,+\infty)\rightarrow X$ starting from $x$. Such $\gamma$ is unique if and only if $u$ is differentiable at $x$. Observe that
\begin{itemize}
  \item any geodesic ray $\gamma:[0,+\infty)\rightarrow X$ starting from $x$ is a $u_{x}-$ray.
\end{itemize}
Since $u_{x}$ is locally semi-concave, there is a one to one correspondence between $\dot{\gamma}(0)$ and elements of $\nabla^{+}u_{x}(x)$.

We argue by contradiction. Assume there is a subset $S\subseteq X$ with positive measure (with respect to the Riemannian volume) such that for every $x\in S$, there exist at least two geodesic rays starting from $x$. We choose any $x_{0}\in S$ and $u_{x_{0}}$ is the point-assigned function associated to $x_0$. Let $Z_{0}$ denote the set of non-differentiable points of $u_{x_{0}}$. By Rademacher theorem, since $u_{x_{0}}$ is 1-Lipschitz, $Z_{0}$ has zero measure. Now we choose any $x_{1}$ in the \textbf{non-empty set} $S\setminus Z_{0}$. By the above choice, $u_{x_{1}}$ is differentiable at $x_{0}$. The discussion at the beginning of the proof implies that $\nabla^{+}u_{x_{1}}(x_1)$ contains at least two elements, therefore $u_{x_{1}}$ is non-differentiable at $x_{1}$. This shows that $u_{x_0}$ and $u_{x_1}$ cannot differ by a constant, which contradicts the condition that the dl-function is unique up to constants.
\end{proof}

Let us look at a simple but interesting example illustrating the situation described in Theorem \ref{uni-criterion}, see also \cite[Page 472]{C}.
\begin{example}
We consider the flat half cylinder $S^{1}\times[0,\infty)$ with the standard product metric $g=d\theta\cdot dr$, where $\theta$ and $r$ denote the coordinate on the unit circle $S^{1}$ and ray $[0,\infty)$ respectively. To make the half cylinder into a non-compact Riemannian manifold, we install a semi-sphere hat $\{(x,y,z)\in\mathbb{R}^{3}|\,\,x^{2}+y^{2}+z^{2}=1,z\geq0\}$ on it by identifying the equator to the boundary $S^{1}\times\{0\}$ of the cylinder, see Figure~\ref{infinitestick} below.

\begin{figure}[h]
   \caption{The infinite stick $X_{0}$}\label{infinitestick}
    \vspace{0.5cm}
  \small \centering
  \includegraphics[width=9cm,height=4cm]{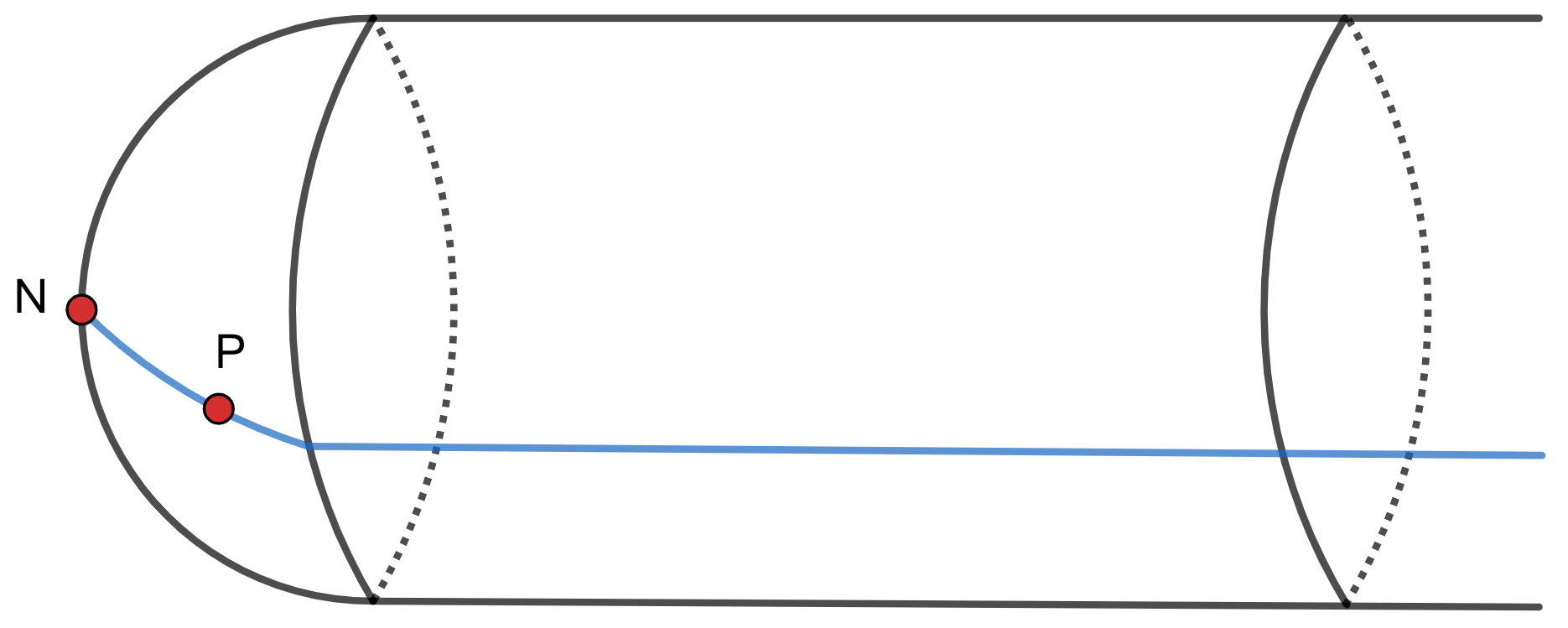}
\end{figure}

We call the resulting manifold an infinite stick and denote it by $X_{0}$. We remark that as a differentiable manifold, $X_{0}$ is of class $C^{1}$ but not $C^{2}$. But by some easy modification, we could make $X_{0}$ into a smooth Riemannian manifold with all the properties listed below unchanged.

\vspace{0.3cm}
Note that $X_{0}$ possesses the following properties:

(1) Except the vertex $N$, there is a unique geodesic ray starting from every point $p$ on $X_{0}$. That is the unique great circle on the semi-sphere connecting $N$ with the equator and passing through $p$ (if $p$ does not locate on the interior of the hat, this part is deleted) continued by a generator $\{\theta=\theta_{0}\}$ of the cylinder, where $(\theta_{0},0)$ is the intersection of the great circle with the equator.

(2) Thus, by Corollary \ref{uni-criterion}, the eikonal equation \eqref{eik} on $X_{0}$ admits only one solution up to constants. On the other hand, it is easy to check that $u(x)=u_{N}(x)=-d(x,N),x\in X_{0}$ is a global solution to \eqref{eik}. So we could conclude that for every $x_{0}\in X_{0}$,
$$
u_{x_{0}}=d(x_{0},N)-d(\cdot,N).
$$
\end{example}

The simplest non-Riemannian space with unique dl-function is the ray with one end-point, other examples can also be constructed.

Every point on these examples admits only one geodesic ray. It is a curious topic to extend the above result to compact, complete, locally compact geodesic spaces which admit, up to constants, only one dl-function. The main difficulty is to avoid the use of measure. In fact, we can ask that
\begin{problem}
Is it possible to find a sufficient and necessary \textbf{geometric condition}, like Theorem \ref{pt-1}, for the uniqueness (up to constants) of the dl-functions?
\end{problem}

\subsection{Level sets: compactness vs non-compactness}
Knowledge about horospheres, the level sets of Busemann function, is proved to be very important in understanding the global geometry of $X$. The same topic for set-assigned dl-functions deserves a discussion.

We consider Hadamard spaces, namely simply connected Riemannian manifold with non-positive curvature. In this situation, 
$$
u_{K}(\cdot)=-d(co K,\cdot),
$$
so the set-assigned dl-function $u_{K}(\cdot)$ attains its maximum and its level sets coincide with sphere $S_r(co K)$ which are compact. Roughly speaking, set-assigned functions measure the distance to all infinities (including any direction). One may ask whether they attain their maxima (or bounded from above) and their level sets are compact in general. However, the following example shows that this is \textbf{not the case}.

\vspace{1em}
To describe our construction, we start with $x$-axis of the Euclidean plane. For each $i\in\mathbb{N}$, we connect $(-i,0)$ with $(i,0)$ by three concatenate segments
\begin{equation}\label{eq-6.1}
\begin{split}
L^{1}_{i}:=\{\,\,(-i,t):t\in[0,i]\,\},\\
L^{2}_{i}:=\{\,\,(t,i):t\in[-i,i]\,\,\},\\
L^{3}_{i}:=\{(i,i-t):t\in[0,i]\}
\end{split}
\end{equation}
and obtain a graph $\mathbf{H}$ on the Euclidean plane (see the following picture):

\begin{figure}[h]
   \caption{The metric space $(H, \hat{d})$}
  \small \centering
  \includegraphics[width=9cm,height=5cm]{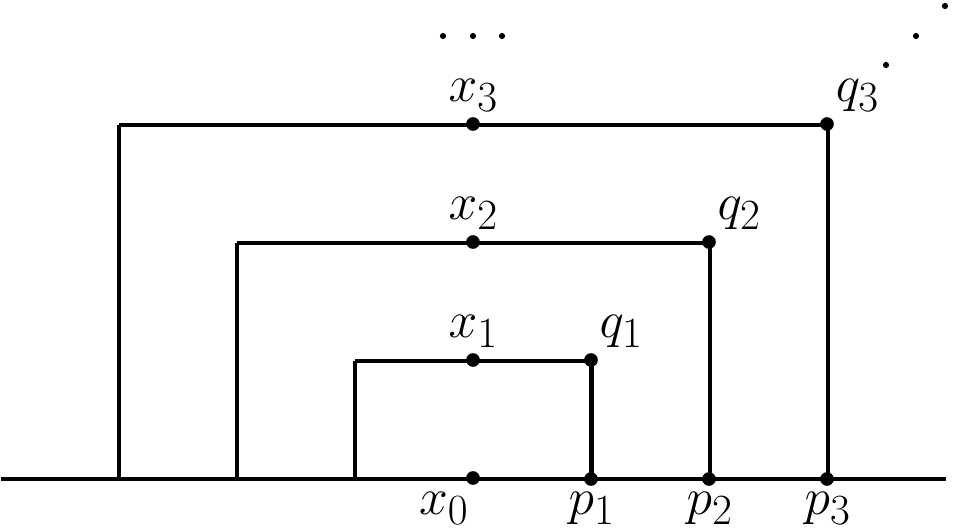}  \label{badexample}
\end{figure}

The Euclidean metric restricted to $\mathbf{H}$ induces a metric $\hat{d}$ \cite[Section 2.3, page 36]{BBI}, which makes $\mathbf{H}$ into a length space. 

For $k\in\mathbb{N}$, let $x_{k}=(0,k), p_{k}=(k,0)$, we shall evaluate $u_{x_{0}}$ at these points. Note that by the definition of $u_{x_{0}}$ and monotonicity of $u^{r}_{x_{0}}(x)$ on $r$, we can take $r$ to be an integer $n\gg k$ and approximate the value $u_{x_{0}}(x)$ by $u^{n}_{x_{0}}(x)$. For convenience, we assume $n\equiv0\,\,(mod\,6)$. In this case,
\begin{align}\label{ex-sp}
S_{n}(x_{0})=\,\,&\{(i,n-i):i\in\mathbb{N},\,\,\frac{n}{2}\leq i\leq n\}\\ \notag
\cup\,\,&\{(3i-n,i):i\in\mathbb{N},\,\,\frac{n}{3}\leq i\leq\frac{n}{2}\}.
\end{align}
From \eqref{ex-sp}, we deduce that for $k<\frac{n}{3}$,
\begin{equation*}
\begin{split}
\hat{d}(p_{k},S_{n}(x_{0}))&=n-k,\\
\hat{d}(x_{k},S_{n}(x_{0}))&=\hat{d}(x_{k},p_{k})+\hat{d}(p_{k},S_{n}(x_{0}))=2k+(n-k)=n+k.
\end{split}
\end{equation*}
Thus by definition of $u_{x_{0}}$, for any $k\in\mathbb{N}$,
\begin{equation}\label{eq-6.2}
\begin{split}
u_{x_{0}}(p_{k})&=\lim_{n\rightarrow\infty}u^{n}_{x_{0}}(p_{k})=-k,\\
u_{x_{0}}(x_{k})&=\lim_{n\rightarrow\infty}u^{n}_{x_{0}}(x_{k})=k.
\end{split}
\end{equation}
Similarly, it is easy to see that for $k\in\mathbb{N}, q_{k}=(k,k)$ are null points of $u_{x_{0}}$. Thus for the length space $(\mathbf{H},\hat{d})$, $u_{x_{0}}$ has no upper bound and its level set $u_{x_{0}}^{-1}(0)$ is not compact.

\end{document}